\documentclass[10pt]{amsart}
\usepackage{amsmath,amsfonts, amssymb,amsthm,amscd,latexsym,euscript,xypic, verbatim}
\usepackage{epsf}
\usepackage{epsfig}
\usepackage{amstext}
\usepackage{xspace}
\usepackage{amsfonts}
\usepackage{amsmath}
\usepackage{amssymb}
\usepackage{xspace}
\usepackage{graphicx}
\usepackage[active]{srcltx}
\usepackage[all]{xy}
\usepackage{graphics}
\usepackage{color}
\usepackage{caption}
\usepackage{subcaption}
\usepackage{blkarray}
\usepackage{multirow}
\swapnumbers
\theoremstyle{plain}
\newtheorem{prop}{Proposition}[section]

\newtheorem{cor}[prop]{Corollary}
\newtheorem{lemma}[prop]{Lemma}
\theoremstyle{definition}
\newtheorem{point}[prop]{}

\newtheorem*{Def*}{Definition}

\newtheorem{example}[prop]{Example}

\newtheorem*{notation*}{Notation}

\newcommand{\calc}{\mathcal C}

\xyoption{2cell}

\begin{document}

\newcommand{\C}{\textbf C}
\newcommand{\N}{\mathbb N}
\newcommand{\Z}{\mathbb Z}
\newcommand{\R}{\mathbb R}

\newcommand{\n}{$n$\nobreakdash}

\title{Combinatorial presentation of multidimensional persistent homology}
\date{\today}
\author{W.~ Chacholski$^1$\quad M.~Scolamiero \quad F.~Vaccarino$^2$}

\address[Wojciech Chach\'{o}lski]{
       Department of Mathematics\\
       KTH\\
       S 10044 Stockholm\\
       Sweden}
\email{wojtek@math.kth.se}

\address[Martina~Scolamiero]{
 	Department of Mathematics\\
	KTH\\
	S 10044 Stockholm\\
         Sweden}
\email{scola@kth.se}

\address[Francesco~Vaccarino]{
	Dipartimento di Scienze Matematiche\\
	Politecnico di Torino\\
	C.so Duca degli Abruzzi n.24,Torino,10129 \\
         Italy.}
\email{francesco.vaccarino@polito.it}
         
\address[Francesco~Vaccarino]{
        	ISI Foundation\\
	Via Alassio 11/c,Torino,10126\\
	Italy.}
\email{vaccarino@isi.it}

\footnotetext[1]{Partially supported by G\"oran Gustafsson Stiftelse and VR grant 2009-6102.}
\footnotetext[2]{Partially supported by the TOPDRIM project funded by the Future and Emerging Technologies program of the European Commission under Contract IST-318121.}

\begin{abstract}
A multifiltration is a functor indexed by $\N^r$ that maps any morphism to a monomorphism. 
The goal of this paper is to describe in  an explicit and combinatorial way the natural  $\N^r$-graded $R[x_1,\ldots, x_r]$-module structure on the homology of a multifiltration of simplicial complexes.
To do that we study multifiltrations of sets and vector spaces.  We prove  
in particular that the  $\N^r$-graded $R[x_1,\ldots, x_r]$-modules that can occur as  $R$-spans of multifiltrations of sets are  the direct sums of monomial ideals.
\end{abstract}

\maketitle

\section{Introduction}
Let $\N^r$ be the poset of $r$-tuples of natural numbers with partial order  given by $(v_1,\ldots,v_r)\leq (w_1,\ldots, w_r)$ if and only if $v_i\leq w_i$ for all $ 1\leq i\leq r$. 
A functor $F\colon\N^r\to \text{Spaces}$, with  values in the category of simplicial complexes, is called a {\bf multifiltration}  if, for any
$v\leq w$ in $\N^r$, the map  $F(v\leq w)\colon F(v)\to F(w)$ is a monomorphism.
Such a multifiltration is called {\bf compact} if $\text{colim}_{\N^r}F$ is a finite complex.
Compact multifiltrations are the main objects we are studying in this article. 
By applying homology with coefficients in a ring $R$ to  $F$  we obtain a functor 
$H_n(F,R)\colon \N^r\to  R\text{-Mod}$ with values in the category of $R$-modules.
 The category of   functors indexed by $\N^r$ with values in $R\text{-Mod}$ is equivalent to the category of
 $\N^r$-graded modules over the polynomial ring $R[x_1,\ldots, x_r]$. 
 The aim of this paper is to describe 
 this $R[x_1,\ldots, x_r]$-module structure on $H_n(F,R)$ in a way that is suitable for calculations.  
 One  very efficient way of doing it would be to give the minimal free presentation of 
 $H_n(F,R)$. This however we are unable to   do directly. Instead we are going to describe 
 two homomorphisms of finitely generated and free $\N^r$-graded $R[x_1,\ldots, x_r]$-modules
${\boldsymbol A}\rightarrow {\boldsymbol B}\rightarrow{\boldsymbol C}$ whose 
 composition  is the zero homomorphism (this sequence is a chain complex), and
$H_n(F,R)$ is isomorphic to  the homology  of this  complex.
 Since the modules  involved are finitely generated and free and the homomorphisms preserve grading, these  homomorphisms are simply
 given by matrices of elements in $R$. In our case  the  coefficients  of the matrices are either $1$, $-1$ or $0$
 and they can be explicitly expressed in terms of the multifiltration (we give a polynomial time procedure of how to do that in Section~\ref{sec funspaces}). One can then use standard computer algebra packages to study algebraic invariants of the module $H_n(F,R)$, in particular one can get its minimal free presentation  as well as a minimal resolution, the set of Betti numbers and the Hilbert function.
These invariants can be used then to study point clouds according to the theory of multidimensional persistence
(see for example~\cite{Multi, Multi2}).  Our procedure reduces the computation of $H_n(F,R)$ to the computation of the homology of a chain complex of free $\N^r$-graded $R[x_1,\ldots,x_r]$-modules.
This is the starting point in~\cite{Multi2} where the authors explain how to calculate this homology in polynomial time. One of our aims has been to show that such calculations can be done effectively 
for  arbitrary compact multifiltrations and not only  for the so called one critical which are
studied in~\cite{Multi2}.

The theory of multidimensional persistence is interesting both from an applied and theoretical point of view. From the applied perspective it is useful to construct algorithms that characterize and distinguish multifiltrations of data sets or networks according to topological features (see~\cite{Hepatic Lesions}). From a theoretical point of view, multidimensional persistence modules are $\N^r$-graded $R[x_1,\ldots,x_r]$-modules built from a multifiltration. It is interesting to study how the combinatorial structure of the multifiltration is reflected in the module structure and that is what we address in this paper. We start with discussing  the multifiltrations of sets in 
Section~\ref{sec funsets}. We recall  the structure of such multifiltrations,  how can they be 
decomposed into indecomposable parts and how to classify the indecomposable  pieces. 
We use it to give an algorithm for producing a free presentation of a multifiltration of sets. 
In Section~\ref{sec funmod} we then study the effect of taking the $R$-span functor on multifiltrations of sets.
We explain why the obtained multifiltrations of $R$-modules are rather special  and  prove that  they correspond to sums of monomial ideals. Since the $R$-span functor commutes with colimits,
 free presentations for multifiltrations of sets can be used to obtain free presentations
 of monomial ideals. These presentations are  used in Section~\ref{sec funspaces} to obtain the desired description of the $R[x_1,\ldots, x_r]$-module structure on $H_n(F,R)$.
 We conclude by pointing out, in Section~\ref{bifiltration}, that for multifiltrations indexed by $\N^2$ a presentation of the module $H_n(F,R)$, as the cokernel of a homogeneous homomorphism, is an easier task. In this case, the kernel of ${\boldsymbol B}\rightarrow{\boldsymbol C}$ is free and therefore, given our previous results, it is sufficient to choose a set of free generators of this kernel to find a presentation of $H_n(F,R)$. The problem of identifying such a set of free generators in an algorithmic and combinatorial way is left as an open question.

 \section{notation}
\begin{point}\label{point categories}
The symbols $\text{Sets}$, $\text{Spaces}$, and  $R\text{-Mod}$ denote the categories of respectively sets, simplicial complexes, and $R$-modules, where we always assume that $R$ is a commutative ring with identity.   The $R$-linear span functor which  assigns to a  set $S$ the free module $R(S)=\bigoplus_S R$ is denoted by $R\colon\text{Sets}\to R\text{-Mod}$. 
\end{point}

\begin{point}\label{point homology}
By definition a {\bf simplicial complex} $X$ is a collection of subsets of 
a setÊ $X_0$ (called the set of vertices of $X$) such that: for any $x$ inÊ $X_0$, $\{x\}\in X$ and if $\sigma\in X$
and $\tau\subset \sigma$, thenÊ $\tau\in X$.  An element $\sigma$ in $X$ is called a {\bf simplex} of dimension $|\sigma|-1$.  A complex is called {\bf finite} if $X_0$ is a finite set. A morphism between two simplicial complexes $f\colon X\to Y$ is by definition
a map of sets  $f\colon X_0\to Y_0$ such that $f(\sigma)$ is a simplex in $Y$ for any simplex $\sigma$ in $X$. A morphism $f\colon X\to Y$ is  a {\bf monomorphism} if and only if
the function $f\colon X_0\to Y_0$ is injective.

Let us choose an  order  $<$ on the set $X_0$. 
For $n\geq 0$, the symbol $X_n$  denotes the set of  strictly increasing sequences $x_0<\cdots <x_{n}$ of elements in $X_0$ for which the subset $\{x_0,\ldots, x_n\}\subset X_0$ is a simplex in $X$. Such a sequence is called an {\bf ordered} simplex of dimension $n$.  For $0\leq i\leq n+1$, by forgetting the $i$-th element in a  sequence
$x_0<\cdots <x_{n+1}$ we get an element in $X_n$.   The obtained map is denoted by
$d_i\colon X_{n+1}\to X_{n}$. By applying the $R$-span functor and taking the alternating sum of the induced maps we obtain:
\[RX_{n+1}\xrightarrow{\partial_{n+1}:=\sum_{i=0}^{n+1}(-1)^i d_i} RX_n
\xrightarrow{\partial_{n}:=\sum_{i=0}^{n}(-1)^i d_i} RX_{n-1}\]
where for  $n=0$, the $R$-module $RX_{-1}$ is taken to be  trivial. 
It is a standard fact that the composition $\partial_n\partial_{n+1}$ is the trivial map
and hence the image $\text{im}(\partial_{n+1})$ is a submodule of the kernel $\text{ker}(\partial_n)$. The quotient $\text{ker}(\partial_n)/\text{im}(\partial_{n+1})$ is called the $n$-th homology of $X$ and is denoted by $H_n(X,R)$.  The isomorphism type of this module  does not depend on the chosen ordering on $X_0$. Note that this is not a functor on the entire category of simplicial complexes.  However if $f\colon X\to Y$ is a monomorphism, then we can choose first an ordering on $Y_0$ and  then use it to induce
an ordering on $X_0$ so the function $f\colon X_0\to Y_0$ is order preserving. With these choices, by applying $f$ to ordered  sequences element-wise, we obtain   a map of sets $f_n\colon X_n\to Y_n$ which commutes with the maps $d_i$. In this way we  get an induced map of homology modules that we denote by $H_n(f,R)\colon H_n(X,R)\to  H_n(Y,R)$.
\end{point}

\begin{point}\label{point polynomials}
The symbol $R[x_1,\ldots ,x_r]$ denotes the $\N^r$-graded polynomial ring
with coefficients in a ring $R$.  The category of $\N^r$-graded $R[x_1,\ldots,x_r]$-modules with the degree preserving homomorphisms is denoted by $R[x_1,\ldots,x_r]\text{-\bf Mod}$ and we use bold face letters to denote such modules.

A monomial in $R[x_1,\ldots,x_r]$ is a polynomial
of the form $x_1^{v_1}\cdots x_r^{v_r}$. Its  grade is given by
$v=(v_1,\ldots ,v_r)$. Such a monomial is also  written as $x^v$. An $\N^r$-graded ideal in  $R[x_1,\ldots,x_r]$ is called {\bf monomial} if it is generated by 
monomials.  An $\N^r$-graded $R[x_1,\ldots,x_r]$-module isomorphic to  the ideal of  $R[x_1,\ldots,x_r]$ generated by a single monomial $x^v$ is called {\bf free on one generator} $v$ and denoted by $< x^v> $.
An $\N^r$-graded $R[x_1,\ldots,x_r]$-module which is isomorphic to a direct sum of free modules on one generator is called {\bf free}. The $R$-module $\text{Hom}(<x^v>,<x^w>)$ is either trivial if $v\not\geq w$, or is isomorphic to $R$ if $v\geq w$.  We use this to identify the $R$-module of homomorphisms between free modules $\text{Hom}(\oplus_{t\in T}<x^{v_t}>,\oplus_{s\in S}<x^{w_s}>)$
with the set of  $S\times T$ matrices of elements in $R$  whose  $(s,t)$  entry is $0$ if
$v_t\not\geq w_s$. Thus  to describe a degree preserving homomorphism between two finitely generated and free  $\N^r$-graded $R[x_1,\ldots,x_r]$-modules we need to specify:
\begin{itemize}
\item A  matrix $M$ of elements in $R$.
\item Two functions, one that assigns to every row of $M$ an element in $\N^r$ and the other
that assigns to every column  of $M$ an element in $\N^r$. The values of these functions are called  grades of the respective rows and columns. The grades of the columns correspond to the grades of the generators  of the domain  of the homomorphism and the grades of the rows correspond to the grades of the generators of the range of the homomorphism.
\item The matrix  $M$ should satisfy the following property: the entry corresponding to a row with grade $w$ and a column with grade $v$  is zero if $v\not\geq w$.
\end{itemize}
\end{point}
\begin{point} 
\label{point colimit}
Let $I$ be a small category. The symbol $\text{Fun}(I,\calc)$ denotes the category of functors indexed by $I$ with values in a category $\calc$ and natural transformations as morphisms.  We use the symbol $\text{Nat}_{\calc}(F,G)$ to denote the set of natural transformations between two functors $F,G\colon I\to \calc$. Recall~\cite{Categories} that the colimit of a functor  $F\colon  I\to \calc$
is an object  $\text{colim}_{I}F$ in $\calc$ together with morphisms $p_i\colon F(i)\to \text{colim}_{I}F$, for any object $i$ in  $I$. These morphisms are required to satisfy the following universal property.
First, for any
 $\alpha\colon i\to j$ in $I$,  $ p_jF(\alpha\colon i\to j)=p_i$. Second,
if $q_i\colon F(i)\to X$  is a sequence of morphisms in $\calc$ indexed by objects of $I$
fulfilling the equality $ q_jF(\alpha\colon i\to j)=q_i$ for any morphism $\alpha$ in $I$, then there is  a unique $f\colon \text{colim}_{ I}F\to X$ such that $q_i= fp_i$ for any object $i$ in $I$.

If $I$ is the empty category, $\text{colim}_{I}F$ is called the {\bf initial} object and denoted by $\emptyset$. The initial object has the property that, for any object $X$ in $\calc$,
the set of morphisms $\text{mor}_C(\emptyset, X)$ has exactly  one element.
If $I$ is a discrete category, then $\text{colim}_{I}F$ is called the {\bf coproduct} and denoted either by $\coprod_{i\in I}F(i)$ or $\bigoplus_{i\in I}F(i)$. The second notation is used only  in the case the coproduct is taken in an additive or abelian category, as for example in $R\text{-Mod}$.
\end{point}

\begin{point}
\label{point decompose}
An object $X$ Êin $\calc$ is called {\bf decomposable} if it is isomorphic to a sum  $X_1\coprod X_2$ where neither $X_1$ nor $X_2$ is  the initial object.
It is {\bf indecomposable} if it is neither  initial nor decomposable.
An object $X$ is  called {\bf uniquely decomposable} if the following two conditions hold.
First, it is isomorphic to a coproduct $\coprod_{i\in I}X_i$ where  $X_i$ is indecomposable for any $i$. Second, if $X$ is isomorphic to $\coprod_{i\in I}X_i$ and
to $\coprod_{j\in J}Y_j$, where $X_i$'s and $Y_j$'s are indecomposable, then there is a bijection $\phi\colon I\to J$ such that
$X_i$ and $Y_{\phi(i)}$ are isomorphic for any $i$ in $I$.

In the category of sets the initial object is the empty set, the coproduct is the disjoint union, 
a set is decomposable if it contains at least two elements, and is indecomposable if it contains exactly one element.
For  $F\colon I\to \text{Sets}$, its colimit is the quotient of $\coprod_{i \in I} F(i)$
by the equivalence relation generated by $x_i$ in $F(i)$ is related to $x_j$ in $F(j)$
if there are morphisms $\alpha\colon i\to k$ and $\beta \colon j\to k$ in $I$
for which $F(\alpha)(x_i)=F(\beta)(x_j)$.

\end{point}
\begin{point}\label{point posetN}
The symbol $\N^r$ denotes  the poset of $r$-tuples of natural numbers with partial order  given by $(v_1,\ldots,v_r)\leq (w_1,\ldots, w_r)$ if and only if $v_i\leq w_i$ for all $ 1\leq i\leq r$. The initial element
$(0,\ldots,0)$ in $\N^r$ is denoted simply by $0$.
 Recall that the partial order on $\N^r$  is a lattice. This means that for any finite set of elements
$S$ in $\N^r$, there are elements $\text{min}(S)$ and $\text{max}(S)$ in  $\N^r$ 
(not necessarily in $S$)
with the following properties. First, for any $v$ in $S$, $\text{min}(S)\leq v\leq \text{max}(S)$. Second, if $u$ and $w$ are elements in $\N^r$ for which $u\leq v\leq w$, for any $ v$ in $S$, then $u\leq \text{min}(S)$ and $\text{max}(S)\leq w$.
Furthermore any non-empty subset $S$ of  $\N^r$ has an element $v$ such that  if $w<v$, then $w$ is not in $S$.  Such elements are called minimal in $S$ and may not be unique.
A  functor $F$ indexed by the poset $\N^r$ that maps any morphism to a monomorphism
is called a {\bf multifiltration}. We will denote the colimit of a functor $F$ indexed by $\N^r$ by $\text{colim}\, F.$
A multifiltration $F\colon\N^r\to \text{Sets}/R\text{-Mod}$ is called {\bf one critical} if for any element
$x$ in $\text{colim}\, F$, the set $\{v\in \N^r\ |\ x\text{ is in the image of }p_v\colon F(v)\to \text{colim}\, F\}$ has a unique minimal element  which we denote by $v_x$ and call the {\bf critical coordinate} of $x$
(see~\cite{Multi2}).
A functor $F\colon\N^r\to \text{Sets}/\text{Spaces}$ is called {\bf compact} if $\text{colim}\, F$ is a finite set/simplicial complex.
\end{point}
\begin{point}\label{point functors}
Let $v$ be an element in $\N^r$. The functor $\text{mor}_{\N^r}(v,-)\colon \N^r\to\text{Sets}$ is called
{\bf free on one generator}.  For example $\text{mor}_{\N^r}(0,-)\colon \N^r\to\text{Sets}$  is the constant functor with value the one point set. Since  $\N^r$ is a poset, the values  of a free functor on one generator are either empty, or the one point set. A functor $F\colon  \N^r\to\text{Sets}$ is called {\bf free} if it is isomorphic to a disjoin  union of free functors on one generator. Note that any free functor is a multifiltration.

Composition with the $R$-span functor $R\colon\text{Sets}\to R\text{-Mod}$, is denoted by the same symbol 
 $R\colon \text{Fun}(\N^r,\text{Sets})\to \text{Fun}(\N^r,R\text{-Mod})$ and called by the same name the $R$-span functor. Recall that this $R$-span functor is the left adjoint to the forget the $R$-module structure functor. This implies that the $R$-span functor commutes with colimits, in particular it maps the initial object to the initial object and commutes with coproducts.

 The functor  $R\text{mor}_{\N^r}(v,-)\colon \N^r\to R\text{-Mod}$ is also called  {\bf free on one generator}.
 A functor $F\colon  \N^r\to R\text{-Mod}$ is called {\bf free} if it is isomorphic to the $R$-span of a free functor
 with values in $\text{Sets}$ or equivalent, if it is isomorphic to a direct sum of free functors on one generator.
\end{point}

\begin{point}\label{point identfunRmod}
Recall that the category of functors  $\text{Fun}(\N^r,R\text{-Mod})$ is equivalent to the category of $\N^r$-graded modules  $R[x_1,\ldots,x_r]\text{-\bf Mod}$. We are going to identify these categories using the following 
explicit equivalence which assigns to a
functor $F\colon \N^r\to R\text{-Mod}$, the $\N^r$-graded $R[x_1,\ldots,x_r]$-module given by  
${\boldsymbol F}:=\oplus_{v\in \N^r}F(v)$, where $x_i$ acts on component $F(v)$ via the map
$F(v\leq v+e_i)$ where $e_i$ is the $i$-th vector in the standard base.
Via this identification, the free functor $R\text{mor}_{\N^r}(v,-)\colon \N^r\to R\text{-Mod}$ is mapped to the
free module $<x^v>$.
\end{point}

\section{Functors with values in Sets}\label{sec funsets}
The aim of this section is to prove several  basic properties of functors of the form
$F\colon  \N^r\to \text{Sets}$. Many of these properties are well known. We start with:
\begin{prop}\label{prop indecomp}
A functor $F\colon  \N^r\to \text{\rm Sets}$ is indecomposable (see~\ref{point decompose}) if and only if 
the set $\text{\rm colim}\, F$   contains exactly one element.
\end{prop}
\begin{proof}
If the values of $F$  are not all empty, then $\text{colim}\, F$ is not empty.
Further more if $F=G\coprod H$, then $\text{colim}\, F=(\text{colim}\, G)\coprod (\text{colim}\, H)$. This shows that if $\text{\rm colim}\, F$  contains exactly one point, then $F$ is indecomposable. On the other hand  we can decompose $F$ as 
$\coprod_{x\in \text{colim}\, F} F[x]$ where,  for any point $x$ in $\text{colim}\, F$, $F[x]\colon \N^r\to \text{Sets}$
is the  subfunctor of $F$  whose values are given by
$F[x](v):=\{y\in F(v)\ |\ p_v(y)=x\}$ (see~\ref{point colimit}).   Observe that not all  the  values of $F[x]$
are empty.   This describes $F$ as a coproduct of indecomposable functors. Thus if $F$ is indecomposable, then $\text{\rm colim}\,  F$ has to contain only one element.
\end{proof}
The argument in the above proof shows  more:
\begin{cor}\label{unique decomposition}
Any functor  $F\colon  \N^r\to \text{\rm Sets}$ is uniquely decomposable as
$F=\coprod_{x\in\text{\rm colim}\, F}F[x]$.
\end{cor}

In this paper we are not interested in all functors indexed by $\N^r$ with values in Sets, but those that map any morphism to a monomorphism. Such functors are called 
multifiltrations of sets (see~\ref{point posetN}) and here is their characterization:
\begin{prop}\label{prop multimono}
A functor $F\colon  \N^r\to \text{\rm Sets}$ is a multifiltration if and only if the map 
$p_v\colon F(v)\to \text{\rm colim}\,  F$ is a  monomorphism for any $v$ in $\N^r$.
\end{prop}
\begin{proof}
Recall that $\text{colim}\, F$ is the quotient of $\coprod_{v\in \N^r} F(v)$ by the
equivalence relation generated by $x_v$ in $F(v)$  is related to $x_w$ in $F(w)$,
if there is $u\geq v$ and $u\geq w$ such that $F(v\leq u)(x_v)=F(w\leq u)(x_w)$.
Note that since $\N^r$  is a lattice, the described relation is  already en equivalence relation. Thus  two elements of $F(v)$ are mapped to the same element in
$\text{colim}\, F$ if and only if they are mapped to the same element via $F(v\leq u)$ for some $u$ and the proposition follows. 
\end{proof}

\begin{cor}\label{cor charindmult}
A functor $F\colon \N^r\to \text{\rm Sets}$ is an indecomposable multifiltration if and only if
the set $F(v)$ has at most one element for any $v$ in $\N^r$ and there is $u$ for which
$F(u)$ is not empty.
\end{cor}
\begin{proof}
Assume first  $F$ is an indecomposable multifiltration. By Proposition~\ref{prop indecomp}, $\text{colim}\,F$  is the one point set. The multifiltration assumption implies that $F(v)$ is a subset of $\text{colim}\,F$ for any $v$ (see~\ref{prop multimono}). Consequently the set $F(v)$  can not contain more than one element. Since $\text{colim}\,F$ is not empty, the values of $F$ can not be all empty either. This shows one implication.

Recall that any element in $\text{colim}\,F$ is of the form $p_v(x)$ for some
$v$ in  $\N^r$ and $x$ in $F(v)$.  Assume  that 
 $\text{colim}\,F$ has at least two  elements, which we write as $p_v(x)$ and $p_w(y)$.  The elements 
 $F(v\leq\text{max}\{v,w\})(x)$ and  $F(w\leq\text{max}\{v,w\})(y)$ therefore also have to be different. Consequently  the set $F(\text{max}\{v,w\})$ has more than one element.  \end{proof}

Indecomposable multifiltrations of sets are therefore exactly the non empty sub-functors of the
free  functor $\text{mor}_{\N^r}(0,-)$ on one generator given by the origin $0$ in $\N^r$ (see~\ref{point functors}).

 Note that since there is a unique map from any set to the one point set, according to~\ref{cor charindmult} , if $F\colon \N^r\to \text{Sets}$ is an indecomposable
multifiltration, then, for any $G\colon \N^r\to \text{Sets}$, there is at most one natural transformation $G\to F$.  Thus the  full subcategory   of $\text{Fun}(\N^r,\text{Sets})$ given by the indecomposable multifiltrations is a poset. 
This is the inclusion poset of all the non empty  sub-functors of the
free  functor $\text{mor}_{\N^r}(0,-)$.
Our next goal is to describe this poset.
We do that using the notion of the {\bf support}
of a functor $F\colon \N^r\to \text{Sets}$:
\[\text{supp}(F):=\{v\in \N^r\ |\ F(v)\not=\emptyset\}\]
For example $\text{supp}(\text{mor}_{\N^r}(v,-))=\{w\in\N^r\ |\ v\leq w\}$.
Not all subsets of $\N^r$ can be a support.  If $v$ belongs to $\text{supp}(F)$,
then so does any $w\geq v$. Subsets of $\N^r$ that satisfy this property are called
{\bf saturated}.

\begin{prop}\label{prop indmultfilt}
The function $(F\colon\N^r\to\text{\rm Sets})\mapsto \text{\rm supp}(F)$ is an isomorphism between
the poset of  indecomposable multifiltrations of sets  and the inclusion poset of saturated non-empty subsets of $\N^r$.
\end{prop}
\begin{proof}
Observe first that if there is a natural transformation $F\to G$, then if $F(v)$ is not empty, then neither is $G(v)$. This means that $\text{supp}(F)\subset\text{supp}(G)$ which shows that the  function
$F\mapsto \text{supp}(F)$  is a function of posets.

To define the inverse of the support function, choose a saturated subset $S$ in $\N^r$ and  an element $v$  in $\N^r$. Set: 
\[\Psi(S)(v):=\begin{cases}
\{v\} &\text{ if } v\in S\\
\emptyset & \text{ if } v\not\in S
\end{cases}
\]
Since $S$ is saturated, if $\Psi(S)(v)$ is not empty, then neither is
$\Psi(S)(w)$ for any $v\leq w$. We can therefore define $\Psi(S)(v\leq w)\colon \Psi(S)(v)\to \Psi(S)(w)$
to be the unique map. This defines a functor which by Corollary~\ref{cor charindmult} is an indecomposable multifiltration. 
The construction $\Psi$ gives a map of posets between the saturated subsets in 
 $\N^r$ and indecomposable multifiltrations.
 
 Note that
 $\text{supp}(\Psi(S))=S$. 
 Furthermore, for any $F\colon \N^r\to \text{Sets}$, there is a unique
 natural transformation $F\to\Psi(\text{supp}(F))$ which becomes an isomorphism if
 $F$ is an indecomposable multifiltration. This shows that $\Psi$ is the inverse of
 the support function.
\end{proof}

Our next step is to  describe the set of saturated subsets of  $\N^r$.
For any subset $S$ of $\N^r$ define $\text{gen}(S):=\{v\in S\ |\ \text{ if } w<v,\text{ then } w\not\in S\}$ and call it the {\bf  minimal set of generators} of $S$. For example $\text{gen}(\text{supp}(\text{mor}_{\N^r}(v,-)))=\{v\}$.  Furthermore~\ref{prop indmultfilt} implies that an indecomposable multifiltration $F\colon\N^r\to\text{Sets}$ is free (necesarily on one generator) if and only if  $\text{gen}(\text{supp}(F))$ consists of one element. This can be generalised to arbitrary multifiltrations:
\begin{prop}\label{prop coneset}
A multifiltration $F\colon\N^r\to\text{Sets}$ is free if and only if it is one critical (see~\ref{point posetN}).
\end{prop}
\begin{proof}
We have a decomposition $F=\coprod_{x\in\text{\rm colim}\, F}F[x]$. Note that
$\text{supp}(F[x])=\{v\in \N^r\ |\ x\text{ is in the image of } p_v\colon F(v)\to \text{\rm colim}\, F\}$.
Thus by definition,  $F$  is one critical if and only if  $\text{gen}(\text{supp}(F[x]))$  
 are one element sets, i.e., if the functors $F[x]$ are free  on one generator, for every $x$ in $\text{colim}\,F$. \end{proof}

Directly from the definition  of the minimal set of generators it follows that:  (1) elements in $\text{\rm gen}(S)$ are not comparable; (2) any element in $S$Ê is comparable to some element in $\text{\rm gen}(S)$.  This first property implies  $\text{gen}(S)$ is   finite, since:

\begin{lemma}\label{lem propgen}
If  $S$ is an infinite subset in $\N^r$, then it contains an infinite chain, i.e., a sequence of the form $v_1<v_2<\cdots$. 
\end{lemma}
\begin{proof}
We argue by induction on $r$. The case $r=1$ is clear since $\N$ is totally ordered.
Assume $r>1$. Consider the projection onto the last $r-1$ components $\text{pr}\colon \N^r\to \N^{r-1}$. If the image $\text{pr}(S)$ is finite, then for some $v$ in $\N^{r-1}$
the intersection  $S\cap \text{pr}^{-1}(v)$ is infinite so it contains an infinite chain as it can be identified with  a subset of $\N$. Assume $\text{pr}(S)$ is infinite.
By induction, it contains an infinite chain $v_1< v_2<\cdots$.  It follows that there is a sequence  of elements in $S$Ê of the form $(a_1,v_1),(a_2,v_2),\ldots$.
Define  $i_1$ to be an  index for which $a_{i_1}=\text{min}\{a_1,a_2\ldots\}$
and set $x_1:=(a_{i_1},v_{i_1})$.  Define $i_2$ to be an  index for which $a_{i_2}=
\text{min}\{a_j\ |\ j>i_1\}$ and set $x_2:=(a_{i_2},v_{i_2})$. Note that $x_1<x_2$. Continue by induction
to obtain a chain $x_1<x_2<\cdots $ in $S$. 
\end{proof}

\begin{prop}\label{prop genofindmultfilt}
The function $S\mapsto \text{\rm gen}(S)$ is a bijection between 
 the set of saturated subsets of $\N^r$ and 
the set of all finite  subsets of  $\N^r$ whose elements are not comparable. 
\end{prop}
\begin{proof}
For a subset $T$Ê in $\N^r$, define:
\[\text{sat}(T):=\{v\ |\ \text{there is } u\text{ in } T
\text{ such that } v\geq u\}\]
We are going  to prove that the function $T\mapsto \text{sat}(T)$  is the inverse to $S\mapsto
\text{gen}(S)$.
Since any element in $S$ is comparable to some element in $\text{gen}(S)$, it follows that 
$S\subset \text{sat}(\text{gen}(S))$. In the case 
$S$ is saturated,  $\text{sat}(\text{gen}(S))\subset  S$ and hence these two sets are equal.

Consider  an element $v$ in $\text{gen}(\text{sat}(T))$.  Since $v$ is in $\text{sat}(T)$,
$u\leq v$ for some $u$Ê in $T$.
 If $u\not=v$, then by definition of $\text{gen}(\text{sat}(T))$, $u$ could not belong to $\text{sat}(T)$, which is a contradiction. Thus $u=v$ and $v$ belongs to $T$. This shows the inclusion $\text{gen}(\text{sat}(T))\subset T$. Assume $T$ consists of non-comparable elements.  Let $v$ be in $T$ and $w<v$. Then $w$ can not belong to
$ \text{sat}(T)$, otherwise, for some $u$ Êin $T$, $u\leq w$ and we would have two comparable elements $v$ and $u$ in $T$.  It follows that  $v$ belongs to $ \text{gen}(\text{sat}(T))$. We can conclude that $T\subset  \text{gen}(\text{sat}(T))$ and hence these two sets are equal.
\end{proof}

\begin{cor}\label{cor posetofideals}
Let $R$ be a commutative ring with a unit.  The poset of  indecomposable multifiltrations of sets is isomorphic to  the  inclusion poset of monomial ideals in
$R[x_1,\ldots, x_r]$.
\end{cor}
\begin{proof}
Let $F\colon \N^r\to \text{Sets}$ be a functor.  Define $\Psi(F)$ to be the monomial ideal in
$R[x_1,\ldots, x_r]$ given by:
\[\Psi(F):=\langle x^v\ |\ v\in \text{gen}(\text{supp}(F))\rangle\]

If there is a  natural transformation $F\to G$, then $\text{supp}(F)\subset \text{supp}(G)$.
We claim that in this case there is an inclusion:
\[\Psi(F)=\langle x^v\ |\ v\in \text{gen}(\text{supp}(F))\rangle\subset
\langle x^v\ |\ v\in \text{gen}(\text{supp}(G))\rangle=\Psi(G)
\]
To see this let $v$ be in $ \text{gen}(\text{supp}(F))$. We show that there is
$u$ in $ \text{gen}(\text{supp}(G))$ such that $u\leq v$. That would imply 
$x^v$ is divisible by $x^u$ proving the claim.
If $v$ belongs   to $\text{gen}(\text{supp}(G))$ there is nothing to prove. 
Assume that this is not the case. Since $v$ belongs to $\text{supp}(G)$,
 there is $u$
in $\text{supp}(G)$ for which indeed $u\leq v$.
In this way  $\Psi$ defines a functor from $\text{Fun}(\N^r,\text{Sets})$ to the inclusion poset
of monomial ideals in $R[x_1,\ldots, x_r]$.
The restriction of $\Psi$ to indecomposable multifiltrations is a  function of posets. 

On the other hand let $I$ be a monomial ideal in $R[x_1,\ldots,x_r ]$, consider the set $S_I:=\{v\in \N^r | x^v \in I\}$.
This is a saturated subset of $\N^r$ because if $u\leq v$ and $x^u$ is in $I$ then $x^v$ must also be in $I$. We define $\Phi(I)$ to be the indecomposable
 multifiltration associated to $S_I$ (see~\ref{prop indmultfilt}). If there is an inclusion of ideals $I \subseteq J$, then $S_I \subseteq S_J$ and  again by~\ref{prop indmultfilt}  
 we have an inclusion $\Phi(I)\subseteq \Phi(J)$. In this way we obtain a functor $\Phi$  between the poset of monomial ideals to the poset of indecomposable multifiltrations. Given a functor $F:\N^r \rightarrow \text{Sets}$ there is a unique natural transformation $F \rightarrow \Phi(\Psi(F))$ and this is an isomorphism if $F$ is an indecomposable multifiltration as both of these functors have the same support. If $I$ is a monomial ideal in $R[x_1,\ldots x_r]$ then it is also immediate to verify that $\Psi(\Phi(I))= I$.
\end{proof}

According to Propositions~\ref{prop indmultfilt} and~\ref{prop genofindmultfilt} the function
 $F\mapsto \text{gen}(\text{supp}(F))$ is a bijection between the set of indecomposable 
 multifiltrations of sets and  finite  non-empty subsets of  $\N^r$ whose elements are not comparable. 
We finish this section with  giving a constructive formula for the inverse to this function.
Let $T$ be a  subset of  $\N^r$. 
Define $F_T\colon \N^r\to \text{Sets}$ to be a functor given by the following coequalizer in $\text{Fun}(\N^r,\text{Sets})$:
\[F_T:=\text{colim}\left(
\xymatrix{\displaystyle{\coprod_{v_0\not=  v_1\in T} \text{mor}_{\N^r}(\text{max}\{v_0,v_1\},-)}
\ar@<1.2ex>[rr]^-{\pi_0}\ar@<-1.2ex>[rr]_-{\pi_1} &&
\displaystyle{\coprod_{v\in T} \text{mor}_{\N^r}(v,-)}}\right)\]
where on the component indexed by $v_0\not=v_1\in  T$, the  map $\pi_i$, 
is given by the unique natural transformation $\text{mor}_{\N^r}(\text{max}\{v_0,v_1\},-)\to \text{mor}_{\N^r}(v_i,-)$
induced by $v_i\leq \text{max}\{v_0,v_1\}$.
\begin{prop}\label{finite set}
If $T\subset \N^r$ is not empty, then
the functor $F_T$ is an indecomposable 
 multifiltration whose support is given by $\text{\rm sat}(T)$.
\end{prop}
\begin{proof}
Let $u$ be an element in $\N^r$. The set $F_T(u)$ Êis a quotient of $\coprod_{v\in T} \text{mor}_{\N^r}(v,u)$ and hence  $F_T(u)\not=\emptyset$ if and only if $\coprod_{v\in T} \text{mor}_{\N^r}(v,u)\not=\emptyset$, implying  the  equality $\text{supp}(F_T)= \text{\rm sat}(T)$. In particular if $T$ Êis non-empty,  then neither is $\text{supp}(F_T)$.

 Let $ v_0\leq u$ and $ v_1\leq u$ be two different elements in $\coprod_{v\in T} \text{mor}_{\N^r}(v,u)$.
 These  inequalities give an element  $\text{max}\{v_0,v_1\}\leq u$ in $\coprod_{v_0\not=v_1\in  T} \text{mor}_{\N^r}(\text{max}\{v_0,v_1\},-)$ which  is  mapped via $\pi_i$ to $v_i\leq u$. The elements $ v_0\leq u$ and $ v_1\leq u$ are therefore sent, via the quotient map,  to the same element in $F_T(u)$. The set  $F_T(u)$ can  therefore have at most one element and hence, according to~\ref{cor charindmult}, $F_T$ is 
 an  indecomposable 
 multifiltration.
\end{proof}
\begin{cor}\label{presentation}
If  $F$  is an indecomposable 
 multifiltration, then it is isomorphic to $F_{\text{\rm gen}(\text{\rm supp}(F))}$.
\end{cor}

We can use the above construction to give a presentation of any multifiltration. Here is a 
procedure of how to do that. Let $F\colon \N^r\to \text{\rm Sets}$ be a multifiltration.  
For any $v$ in $\N^r$, index  elements of $F(v)$ by elements of $\text{colim}\, F$ as follows: $y$ in $F(v)$
has index $x$ in $\text{colim}\, F$ if $p_v(y)=x$. Let $F[x]$ be the subfunctor of $F$ whose elements
have index $x\in\text{colim}\, F$ (see the proof of~\ref{prop indecomp}). It is an indecomposable multifiltration. Recall that $F=\coprod_{x\in\text{colim}\, F}F[x]$.
 The functor  $F$ is then isomorphic to:
 \[\coprod_{x\in\text{\rm colim}\, F} F_{\text{\rm gen}(\text{\rm supp}(F[x]))}\]

Since we are going to use this presentation, we need to introduce notation describing the involved functors. 

\begin{itemize}
\item  For any $x$ in $\text{colim}\, F$, define:
\[\mathcal{G}F[x]:=\coprod_{v \in  \text{\rm gen}(\text{\rm supp}(F[x]))}\text{mor}_{\N^r}(v,-)\]
\[\mathcal{K}F[x]:=\coprod_{v_0\not=v_1 \in  \text{\rm gen}(\text{\rm supp}(F[x]))}
 \text{mor}_{\N^r}(\text{max}\{v_0,v_1\},-)
\]
\item Recall that  there are  natural transformations $\pi_0[x],\pi_1[x]\colon\mathcal{K}F[x]\to 
\mathcal{G}F[x]$ induced by $v_0\leq \text{max}\{v_0,v_1\}$ and $v_1\leq \text{max}\{v_0,v_1\}$.
\item Since $F[x]$ is  indecomposable, there is a unique natural transformation denoted by $p_{F,x}\colon \mathcal{G}F[x]\to  F[x]$.  This natural transformation has the universal property describing
$F[x]$ as the colimit of the diagram:
\[\xymatrix@C=50pt{\mathcal{K}F[x]\ar@<1.2ex>[r]^-{\pi_0[x]}\ar@<-1.2ex>[r]_-{\pi_1[x]}  &\mathcal{G}F[x]}\]
\end{itemize}

By summing over all $x$ in $\text{colim}\, F$, we obtain  functors
$\mathcal{G}F\colon=\coprod_{x\in \text{colim}\, F} \mathcal{G}F[x]$,  $\mathcal{K}F\colon=\coprod_{x\in \text{colim}\, F} \mathcal{K}F[x]$ and  natural transformations
$\pi_0,\pi_1\colon  \mathcal{K}F\to  \mathcal{G}F$ and $p_F\colon =\coprod_{x \in \text{colim}\, F} p_{F,x}
\colon \mathcal{G}F\to F$.  The  natural transformation $p_F$ has the universal property describing
$F$ as the colimit of the diagram:
\[\xymatrix@C=50pt{\mathcal{K}F\ar@<1.2ex>[r]^-{\pi_0}\ar@<-1.2ex>[r]_-{\pi_1}  &\mathcal{G}F}\]

Although the natural transformations $p_{F,x}$ are unique, the construction $F\mapsto \mathcal{G}F$ is not functorial. Nevertheless we attempt to define it also for a natural transformation $\alpha\colon F\to G$. Consider the  map of sets $\text{colim}\, \alpha\colon \text{colim}\,F \to \text{colim}\,G$. Since for any $v$ in $\N^r$, the following square commutes, we get an inclusion $\alpha(F[x])\subseteq G[\text{colim}\,  \alpha(x)]$
\[
\xymatrix@R=17pt@C=35pt
{F(v) \ar[r]^-{ \alpha (v)}\ar[d]_{p_v} & G(v) \ar[d]^{p_v} \\
\text{colim}\,  F \ar[r]^-{ \text{colim}\, \alpha} & \text{colim}\,  G}
\]
It follows that  the set
$\{w\in  \text{\rm gen}(\text{\rm supp}(G[\text{colim}\, \alpha(x)])\ |\ w\leq v\}$ is not empty for any  $v$ in  $\text{\rm gen}(\text{\rm supp}(F[x]))$.
We can order this set using the lexicographical  order and  define
$w_{\alpha,x,v}$ to be the smallest element of this set. Since $w_{\alpha,x,v}\leq v$, there is a unique natural transformation $\text{mor}_{\N^r}(v,-)\to\text{mor}_{\N^r}(w_{\alpha,x,v},-)$.
Define $\overline{\alpha}\colon \mathcal{G}F\to \mathcal{G}G$ to be the natural transformation
which on the summand $\text{mor}_{\N^r}(v,-)$ 
indexed by  $x$ in $\text{colim}\, F$ and 
$v$ in $\text{gen}(\text{supp}(F[x]))$ is given by the composition of $\text{mor}_{\N^r}(v,-)\to\text{mor}_{\N^r}(w_{\alpha,x,v},-)$ and the inclusion into  $\mathcal{G}G$ of  the summand  $\text{mor}_{\N^r}(w_{\alpha,x,v},-)$   indexed by $ \text{colim}\, \alpha(x)$  in  $ \text{colim}\, G$ and  
$w_{\alpha,x,v}$ in $\text{\rm gen}(\text{\rm supp}(G[\text{colim}\, \alpha(x)])$.
Because of these choices we obtain a commutative diagram of natural transformations:
\[\xymatrix{\mathcal{G}F\rto^{\overline\alpha} \dto_{p_F}& \mathcal{G}G\dto^{p_G}\\
F\rto^{\alpha} & G}\]
Explicitly:
\[\xymatrix@R=35pt@C=35pt{
\text{mor}_{\N^r}(v,-)\ar@{^(->}[d]|{\text{summand indexed by $x$ and $v$}}\rto &
\text{mor}_{\N^r}(w_{\alpha,x,v},-)\ar@{^(->}[d]|{\text{summand indexed by $\text{colim}\, \alpha(x)$ and $w_{\alpha,x,v}$}} \\
\displaystyle{\coprod_{x \in \text{colim}\,F}\coprod_{v \in  \text{gen}(\text{supp}(F[x]))}\text{mor}_{\N^r}(v,-)} \ar[r]^-{ \overline{\alpha}}\ar[d]_-{p_F} &\displaystyle{\coprod_{x \in \text{colim}\, G}\coprod_{v \in  \text{gen}(\text{supp}(G[x]))}\text{mor}_{\N^r}(v,-)}\ar[d]^-{p_G} \\
\coprod_{x \in \text{colim}\, F} F[x] \ar[r]^-{\alpha} & \coprod_{x \in \text{colim}\, G} G[x] }
\] 
It is important to  point out that the assignment $(\alpha\colon F\to G)\mapsto (\overline{\alpha}\colon \mathcal{G}F\to \mathcal{G}G)$ is not a functor. It is not true in general that $\overline{\beta\, \alpha}$
equals $\overline{\beta}\,\overline{\alpha}$.

\section{Set valued vs.\ $R\text{-Mod}$ valued  functors}\label{sec funmod}
Let $R$ be a commutative ring with identity.
Recall that we identify the category of functors 
$\text{Fun}(\N^r,R\text{-Mod})$ with  the category of $\N^r$-graded $R[x_1,\ldots,x_r]$-modules by  assigning  to $F\colon\N^r\to R\text{-Mod}$ the 
$\N^r$-graded $R[x_1,\ldots,x_r]$-module  given by ${\boldsymbol F}=\oplus_{v\in \N^r}F(v)$  (see~\ref{point identfunRmod}).
Via the above identification the free  functor $R\text{mor}_{\N^r}(0,-)$ (see~\ref{point functors}) is mapped to  
the module $R[x_1,\ldots,x_r]$. Thus  sub-functors of $R\text{mor}_{\N^r}(0,-)$ are identified with
 $\N^r$-graded ideals in $R[x_1,\ldots,x_r]$.   Among these sub-functors  there are
 the $R$-spans of indecomposable multifiltrations of sets and among the  $\N^r$-graded ideals in $R[x_1,\ldots,x_r]$ there are the monomial ideals. Note that for an indecomposable multifiltration of sets $F\colon\N^r\to\text{Sets}$, 
 the $\N^r$-graded ideal  $\boldsymbol{RF}\subset R[x_1,\ldots,x_r]$ coincides with the monomial ideal  $\Psi(F)$  given in the proof of Corollary~\ref{cor posetofideals}. It thus follows from this corollary that the sub-functors of $R\text{mor}_{\N^r}(0,-)$ that are identified with  monomial ideals are exactly the $R$-spans of indecomposable multifiltrations of sets. Since 
 monomial ideals are indecomposable $R[x_1,\ldots,x_r]$-modules, then so are the $R$-spans of indecomposable multifiltrations of sets. These are the easiest indecomposable multifiltrations of $R$-modules.
 The following is a key fact  about  their finite sums:
 
 \begin{prop}\label{prop sepciamultifltofmod}
 Let $\{F_i\colon\N^r\to \text{\rm Sets}\}_{1\leq i\leq n}$ and $\{G_j\colon\N^r\to \text{\rm Sets}\}_{1\leq j\leq m}$ be two finite families of  indecomposable multifiltrations of sets.
 If $\oplus_{i=1}^{n}RF_i\colon\N^r\to R\text{\rm -Mod}$ and  $\oplus_{j=1}^{m}RG_j\colon\N^r\to R\text{\rm -Mod}$ are isomorphic as functors with values  in $R\text{\rm -Mod}$, then $n=m$, and there is a permutation $\sigma$  of $\{1,\ldots, n\}$ for which $F_i\colon \N^r\to\text{\rm Sets}$ and $G_{\sigma(i)}\colon \N^r\to\text{\rm Sets}$  are isomorphic for anyÊ $i$.
 \end{prop}
 \begin{proof}
 First note that if $F,G\colon \N^r\to \text{Sets}$ are indecomposable multifiltrations, then
 the map $R\text{Nat}_{\text{Sets}}(F,G)\to \text{Nat}_{R\text{-Mod}}(RF,RG)$,
 induced by the $R$-span functor, is an isomorphism of $R$Ê modules (this is not true  if $F$ is a  multifiltration but not indecomposable). 
 Consequently  the $R$ module $\text{Nat}_{R\text{-Mod}}(RF,RG)$ is isomorphic to $R$ 
 if $\text{supp}(F)\subset \text{supp}(G)$ or it is trivial if $\text{supp}(F)\not\subset \text{supp}(G)$.
 
 We proceed by induction on $n$ to prove the proposition.
 Assume $n=1$. Since $RF$ and $\oplus_{j=1}^{m}RG_j$ are isomorphic, then so are their colimits which as $R$ modules are isomorphic to respectively $R$ and $\oplus_{j=1}^{m}R$. For commutative rings the rank of a free module is a well define invariant and hence $m=1$. The functors $RF$ and  $RG_1$ are therefore  isomorphic and  by the discussion above $\text{supp}(F)$ and $\text{supp}(G)$ are the same subsets of $\N^r$. We can then use~\ref{prop indmultfilt} to get Ê $F$ and $G$ are isomorphic.
 
Assume $n>1$. Consider the subsets $\text{supp}(F_i)\subset \N^r$ for $1\leq i\leq n$ and choose among them a maximal one  $T$ with respect to the inclusion. By permuting we can assume that:
 \[\text{supp}(F_i) = T  \text{, if } 1\leq i\leq n'\ \ \ \ \ \text{and}  \ \ \ \ \ 
\text{supp}(F_i)  \not= T \text{, if } n'<i\leq n
\]
Let $\phi\colon\oplus_{i=1}^{n}RF_i\to \oplus_{j=1}^{m}RG_j$ and 
$\psi\colon \oplus_{j=1}^{m}RG_j\to\oplus_{i=1}^{n}RF_i $ be inverse isomorphisms.
Since the restriction of $\phi$ to $F_1$ is  non trivial, there is $j$ such that
$T=\text{supp}(F_1)\subset\text{supp}(G_j)$. By the same argument, since the restriction of $\psi$ to $G_j$ is not  trivial, there is $l$ for which
$ \text{supp}(G_j)\subset \text{supp}(F_l)$. As we chose $T$ to be a maximal among
the supports of $F_i$'s, we get $l\leq n'$ and $ \text{supp}(G_j)=T$.  Again by permuting if necessary we can assume that:
\[T=\text{supp}(G_i)   \text{, if } 1\leq i\leq m'\ \ \ \ \ \text{and}  \ \ \ \ \ 
T\not\subset \text{supp}(G_i) \text{, if } m'<i\leq m
\]
This means that $\phi $ maps the submodule $\oplus_{i=1}^{n'}RF_i \subset 
\oplus_{i=1}^{n}RF_i$ to the submodule $\oplus_{j=1}^{m'}RG_j \subset 
\oplus_{j=1}^{m}RG_j$. Furthermore the restriction of $\phi\colon \oplus_{i=1}^{n'}RF_i\to \oplus_{j=1}^{m'}RG_j$ is an isomorphism whose inverse is given by the restriction of $\psi$.    We therefore get that their colimits $\oplus_{i=1}^{n'} R$ and $\oplus_{j=1}^{m'} R$ are also isomorphic and hence $n'=m'$. Moreover, 
by taking the quotients, we obtain an isomorphism between $\oplus_{i>n'}^{n}RF_i$ and
$\oplus_{j>n'}^{m}RG_j$. The proposition now follows from the inductive assumption.
 \end{proof}
 
 The above proposition can be restated in the form:
 \begin{cor}\label{cor isotypeofmultofsets}\hspace{1mm}
 \begin{enumerate}
 \item Let $\{I_i\}_{1\leq i\leq n}$ and $\{J_j\}_{1\leq j\leq m}$ be monomial ideals
 in $R[x_1\ldots,x_r]$. If the $\N^r$-graded $R[x_1,\ldots, x_r]$ modules $\oplus_{i=1}^n I_i$ and 
 $\oplus_{j=1}^m J_j$  are isomorphic, then $m=n$ and 
 there is a permutation $\sigma$  of $\{1,\ldots, n\}$ for which $I_i=J_{\sigma(i)}$.
 \item
 Let $F,G\colon\N^r\to \text{\rm Sets}$ be compact multifiltrations (see~\ref{point posetN}).  Then 
 $F$ and $G$ are isomorphic if and only if their $R$-spans $RF,RG\colon\N^r\to R\text{\rm -Mod}$
 are isomorphic.
 \end{enumerate}
 \end{cor}
 
 The statement~\ref{cor isotypeofmultofsets}.(2) is not true 
if the  functors $F$ and $G$  are not multifiltrations:
   \begin{example}\label{notfromsets}
Let $F_1,F_2\colon \N\to \text{Sets}$ be functors with the same values  $F_1(0)=F_2(0)=\{a,b,c,d\}$,
$F_1(1)=F_2(1)=\{e,f\}$ and $F_1(n)=F_2(n)=\{g\}$ for $n\geq 2$, however with different maps which are given by the following diagrams: 
\[\xymatrix@R=5pt@C=20pt{
F_1(0)\rto & F_1(1)\rto & F_1(2)\\
a\ar@{|->}[dr] \\
b\ar@{|->}[r]  & e\ar@{|->}[r]  & g\\
c\ar@{|->}[r]  & f\ar@{|->}[ur] \\
d\ar@{|->}[ur] 
}\hspace{2cm}
\xymatrix@R=5pt@C=20pt{
F_2(0)\rto & F_2(1)\rto & F_2(2)\\
a\ar@{|->}[dr] \\
b\ar@{|->}[r]  & e\ar@{|->}[r]  & g\\
c\ar@{|->}[ur]  & f\ar@{|->}[ur] \\
d\ar@{|->}[ur] 
}\]
Although the functors $F_1$ and $F_2$ are not isomorphic, their $R$-spans $RF_1$ and $RF_2$ are.
\end{example}

The following example illustrates the fact that not  all (indecomposable) multifiltrations of $R$-modules are $R$-spans of (indecomposable) multifiltrations of sets.  

\begin{example}
 Consider the  
 multifiltration  $F:\N^2\rightarrow R\text {-Mod}$ which on the square $\{ v\leq (2,2)\}\subset \N^2$ is given by the following commutative  diagram:
 \[\xymatrix{
 R\rto^-{\alpha} & R\oplus R \rto^-{\text{id}} & R\oplus R\\
 0\uto\rto & R\uto_-{\beta} \rto^-{\beta} & R\oplus R\uto_{\text{id}}\\
 0\rto\uto & 0\uto\rto & R\uto_-{\gamma}
 }\]
 and for $w$ in $\N^2\setminus \{ v\leq (2,2)\}$, the map
 $F(\text{min}(w,(2,2))\leq w)$ is an isomorphism.
 Assume further that $\alpha$, $\beta$, and $\gamma$ are monomorphisms and their  images  are pairwise  different submodules of $R\oplus R$. Then this functor is not isomorphic to the  $R$-span of any functor with values in $\text{Sets}$.
 Note further that in this case  $F$ is an indecomposable  multifiltration of $R$-modules whose colimit is free of rank $2$ (compare with~\ref{prop indecomp}).
 \end{example}

Being one critical (see~\ref{point posetN}) for multifiltrations of sets is equivalent to being free (see~\ref{prop coneset}). This is not true
for multifiltrations of $R$-modules if $r>2$:
\begin{example}
Consider the multi filtration $F\colon\N^3\to R\text{-Mod}$ which on the cube $\{v\leq (1,1,1)\}\subset \N^3$ is given by the following commutative  diagram:
\[\xymatrix@C=15pt@R=15pt{
& R^2\rrto^-{\alpha} & & R^4\\
0\rrto\urto & & R^2\urto^(.4){\beta}\\
& 0\uuto|\hole\rrto|\hole & & R^2\uuto_-{\gamma}\\
0\uuto\rrto\urto & & 0\uuto\urto
}\]
and  for $w$ in $\N^3\setminus \{v\leq (1,1,1)\}$ the map
$F(\text{min}(w,(1,1,1))\leq w)$
is an isomorphism.
Then this functor is one critical, 
 it is not free,  and it is not  the $R$-span of a multiflitration of sets. 
\end{example}
For bifiltrations ($r=2$) we have the following positive result:

\begin{prop}
Assume $R$ is a field. 
A bifiltration $F\colon\N^2\to R\text{\rm -Mod}$ is free if and only  if it is one critical.
\end{prop}
\begin{proof}
One implication holds more generally for all $r$.
If $F\colon\N^r\to R\text{\rm -Mod}$ is free, it is the $R$-span of a free functor $G: \N^r \rightarrow \text{Sets}$.
Thus  $F$ Êis isomorphic to $\bigoplus_{x\in \text{colim}\, G} R\text{mor}(v_x,-)$.
Since the $R$-span functor commutes with colimits, we can identify $\text{colim}\, F$ with $R(\text{colim}\, G)$.
Consider  an element $y=\sum_{i=1}^{n} c_i x_i$ in $\text{colim}\, F$ where $x_i$ belongs to $\text{colim}\, G$. 
Note that:
\[\{v\in \N^2\ |\ y\in F(v)\}=\bigcap_{i=1}^{n}\{v\in \N^2\ |\ x_i\in G(v)\}\] 
It follows that this set has a unique minimal
element given by $\text{max}\{v_{x_i}\ |\ 1\leq i\leq n\}$. This shows that $F$Ê is one critical.

Assume now that $F\colon\N^2\to R\text{\rm -Mod}$ is one critical. To show that it is free it would be enough to prove that it is the $R$-span of a multifiltration of sets since in this case this multifiltration of sets 
would be also one critical and therefore free by~\ref{prop coneset}.  
Define $G(0,0)$ to be a base of $F(0,0)$.  Since $F((0,0)\leq (1,0))\colon F(0,0)\to F(1,0)$ is an inclusion, we can  extend that base of $F(0,0)$ to a base $G(1,0)$
of $F(1,0)$. We can proceed by induction on $n$ and define in this way a sequence of sets
\[G(0,0)\subset G(1,0)\subset\cdots G(n,0)\subset\cdots\] whose $R$-span gives the functor $F$ restricted to
$\N\times \{0\}\subset \N^2$. We continue again by induction. Assume that $k>1$ and   we have constructed a functor: 
\[G\colon \N\times \{v\in \N\ |\ v<k\}\to \text{Sets}\] whose $R$-span  is isomorphic to the restriction  of $F$. By the same argument as before, since $F((0,k-1)\leq (0,k))\colon F(0,k-1)\to F(0,k)$ is an inclusion
we can extend the base $G(0,k-1)$ of $F(0,k-1)$ to a base $G(0,k)$ of $F(0,k)$. Assume $n>0$ and that we have defined a functor:
\[G\colon  \N\times \{v\in \N\ |\ v<k\}\cup  \{v\in \N\ |\ v<n\}\times \{v\in \N\ |\ v\leq k\}\to \text{Sets}\]
whose $R$-span is isomorphic to the restriction of $F$.
Since $F$ is one critical the intersection of the images of $F(n-1,k)$ and $F(n,k-1)$ in $F(n,k)$ coincide
with the image of $F(n-1,k-1)$. It follows that the induced map:
\[\text{colim}(F(n,k-1)\hookleftarrow F(n-1,k-1)\hookrightarrow
F(n-1,k))\to F(n,k)\] 
is an inclusion. Here the assumption $r=2$ is crucial.
We can then extend the subset:
\[\text{colim}(G(n,k-1)\hookleftarrow G(n-1,k-1)\hookrightarrow
G(n-1,k))\hookrightarrow F(n,k)\] 
to a base $G(n,k)$ of $F(n,k)$. In this way we get a desired functor 
\[G\colon \N\times \{v\in \N\ |\ v\leq k\}\to \text{Sets}\]
whose $R$-span is isomorphic to $F$.
\end{proof}

We finish this section with  a procedure of obtaining a free presentation of the $\N^r$-graded $R[x_1,\ldots,x_r]$-module $\boldsymbol{RF}$ associated to the $R$-span of a  multifiltration
$F\colon\N^r\to \text{Sets}$.
In the first 3 steps we recall from the end of Section~\ref{sec funsets} how to build a presentation of $F$.
\begin{itemize}
\item Decompose $F$ into indecomposable components $\coprod_{x\in \text{colim}\, F} F[x]$.
\item For any $x$, find the set $T_x:=\text{gen}(\text{supp}(F[x]))$.
\item  Recall thatÊ $F$ can be described as the coequalizer of two natural transformations
  $\pi_0,\pi_1\colon\mathcal{K}F\to\mathcal{G}F$ between free functors. Explicitly $F$ is isomorphic to the colimit of  the following diagram:
\[
\coprod_{x\in \text{colim}\, F} \left(
 \vcenter{\vbox{\xymatrix{
\displaystyle{\coprod_{v_0\not=  v_1\in T_x} \text{mor}_{\N^r}(\text{max}\{v_0,v_1\},-)}
\ar@<1.2ex>[rr]^-{\pi_0[x]}\ar@<-1.2ex>[rr]_-{\pi_1[x]} &&
\displaystyle{\coprod_{v\in T_x} \text{mor}_{\N^r}(v,-)}} 
}} \right)
\]
where on the component indexed by $v_0\not=v_1\in  T_x$, the  map $\pi_i$, 
is given by the unique natural transformation $\text{mor}_{\N^r}(\text{max}\{v_0,v_1\},-)\to \text{mor}_{\N^r}(v_i,-)$
induced by $v_i\leq \text{max}\{v_0,v_1\}$.
\item Since the $R$-span functor commutes with colimits, we get that the module  $\boldsymbol{RF}$ is isomorphic to
the  coequalizer of the following two maps $\boldsymbol{\pi_0}$ and $\boldsymbol{\pi_1}$ between free $\N^r$-graded $R[x_1,\ldots,x_r]$-modules (see~\ref{point polynomials}):
\[
\bigoplus_{x\in \text{colim}\, F} \left(\xymatrix{\displaystyle{\bigoplus_{v_0\not=  v_1\in T_x} <x^{\text{max}\{v_0,v_1\}}>}
\ar@<1.2ex>[rr]^-{\boldsymbol{\pi_0[x]}}\ar@<-1.2ex>[rr]_-{\boldsymbol{\pi_1[x]}} &&
\displaystyle{\bigoplus_{v\in T_x} <x^v>}}  \right)
\]
where  $\boldsymbol{\pi_i[x]}$, on the component indexed by $v_0\not=v_1\in  T_x$,  is    given by the inclusion
$<x^{\text{max}\{v_0,v_1\}}>\hookrightarrow  <x^{v_i}>$. Thus the columns of the matrix representing $\boldsymbol{\pi_i[x]}$  have all  entries zero except one which is one.
\item   The  module   $\boldsymbol{RF}$ is then isomorphic to the cokernel of the difference $\boldsymbol{\pi_0}-\boldsymbol{\pi_1}$.
Note that the columns of the matrix $M(F_x)$ representing $\boldsymbol{\pi_0}-\boldsymbol{\pi_1}$  are vectors of the  form: one entry is $1$, one entry is $-1$, and all other entires are zero. 
\end{itemize}

To summarize,
with a multifiltration $F\colon\N^r\to \text{Sets}$ we have associated the following invariants:
\begin{enumerate}
\item a set $\text{colim}\, F$;
\item for any $x$ in $\text{colim}\, F$, a finite subset $T_x:=\text{gen}(\text{supp}(F[x]))$ of $\N^r$;
\item  for any $x$ in $\text{colim}\, F$, 
 a $|T_x|\times{ |T_x|\choose 2}$ matrix  $M(F_x)$, representing the map $\boldsymbol{\pi_0[x]}-\boldsymbol{\pi_1[x]}$ whose columns are vectors of the  form: one entry is $1$, one entry is $-1$, and all other entires are zero. 
\end{enumerate}
These invariants can be used to get the $\N^r$-graded $R[x_1,\ldots,x_r]$-module associated to  the $R$-span $RF$ as the cokernel of the map:
\[\bigoplus_{x\in \text{colim}\, F}\left(
\bigoplus_{v_0\not=  v_1\in T_x} <x^{\text{max}\{v_0,v_1\}}>
\xrightarrow{M(F_x)}
\bigoplus_{v\in T_x} \langle x^v \rangle
\right)\]

\section{Functors with values in $\text{\rm Spaces}$}\label{sec funspaces}
Let $F:\N^r\rightarrow \text{\rm Spaces}$ be a  multifiltration of simplicial complexes,
$X:=\text{colim}\, F$, and $R$ a commutative ring with identity.   Let us choose an ordering on the set of vertices of $X$. 
Since $F$ is a multifiltration, we can restrict this ordering to the set of vertices of $F(v)$, for any $v$ in $\N^r$. In this way the maps $F(v\leq w)$ are order preserving and we can  form a  functor of ordered $n$-simplices to get a multifiltration of sets
$F_n\colon\N^r\to \text{Sets}$ (see~\ref{point homology}) which assigns to any $v$ in $\N^r$ the set $F(v)_n$ of ordered $n$-simplices in $F(v)$. 
These functors, for various $n$'s, are connected via the natural transformations
given by the maps
$d_i: F_{n+1}(v)\rightarrow F_{n}(v)$ which forget the $i$-th element of an ordered simplex (see~\ref{point homology}).
By applying the $R$-span functor and taking the alternating sum of the induced maps we obtain a diagram of natural transformations in $\text{Fun}(\N^r,R\text{-Mod})$:
\[RF_{n+1}\xrightarrow{\partial_{n+1}:=\sum_{i=0}^{n+1}(-1)^i d_i} RF_n
\xrightarrow{\partial_{n}:=\sum_{i=0}^{n}(-1)^i d_i} RF_{n-1}\]
The composition of these maps is trivial and hence we can form a homology functor
$H_n(F,R)\colon \N^r\to R\text{-Mod}$ which in general may not be a  multifiltration.
This could be done in two stages. First we could take the cokernel
of the first differential $\text{coker}(\partial_{n+1}\colon RF_{n+1}\to RF_{n})$
and then the kernel of the induced map $\partial_n\colon \text{coker}(\partial_{n+1})\to 
RF_{n-1}$ or we could take the kernel of the second differential $\text{ker}(\partial_{n}\colon RF_{n}\to RF_{n-1})$ and then the cokernel of the induced map
$\partial_{n+1}\colon RF_{n+1}\to \text{ker}(\partial_{n})$.
Let us consider  the case of $n=0$. Recall that since  $RF_{-1} $ is assumed to be the trivial functor (see~\ref{point homology}), $H_0(F,R)$ is given
by the cokernel $\text{coker}(d_0-d_1\colon RF_{1}\to RF_0 )$. This cokernel is simply the coequalizer
of the two maps $d_0,d_1\colon RF_{1}\to RF_0 $. As the $R$-span functor commutes with colimits,
we then get an isomorphism between $H_0(F,R)$ and the $R$-span of the following functor with values in the category of sets:  
\[ \text{colim}\left(
\xymatrix{
F_1\ar@<1.2ex>[r]^-{d_0}\ar@<-1.2ex>[r]_-{d_1}  & F_0}
\right)
\]
This is a special property of the $0$-th homology.
If  $n\geq 1$, then  it is not true  in general that    the functors $H_n(F,R)$,
$\text{coker}(\partial_{n+1}\colon RF_{n+1}\to RF_{n})$, and $\text{ker}(\partial_{n}\colon RF_{n}\to RF_{n-1})$ are $R$-spans of functors with values in the category of sets, even
 if $R$ is  a field as the following example illustrates:

\begin{example}\label{multifiltrations example}
Consider the two multifiltrations of spaces $F,G:\N^2\rightarrow \text{Spaces}$ which on the square $\{ v\leq (2,2)\}\subset \N^2$ are described in Figure \ref{multifiltrations} and 
for $w$ in $\N^2\setminus \{ v\leq (2,2)\}$, the  maps induced by  $(\text{min}\{w_1,2\},\text{min}\{w_2,2\})\leq w$ are the identities.

\begin{figure}[ht]
        \centering
        \begin{subfigure}[b]{0.3\textwidth}
                \includegraphics[width=\textwidth]{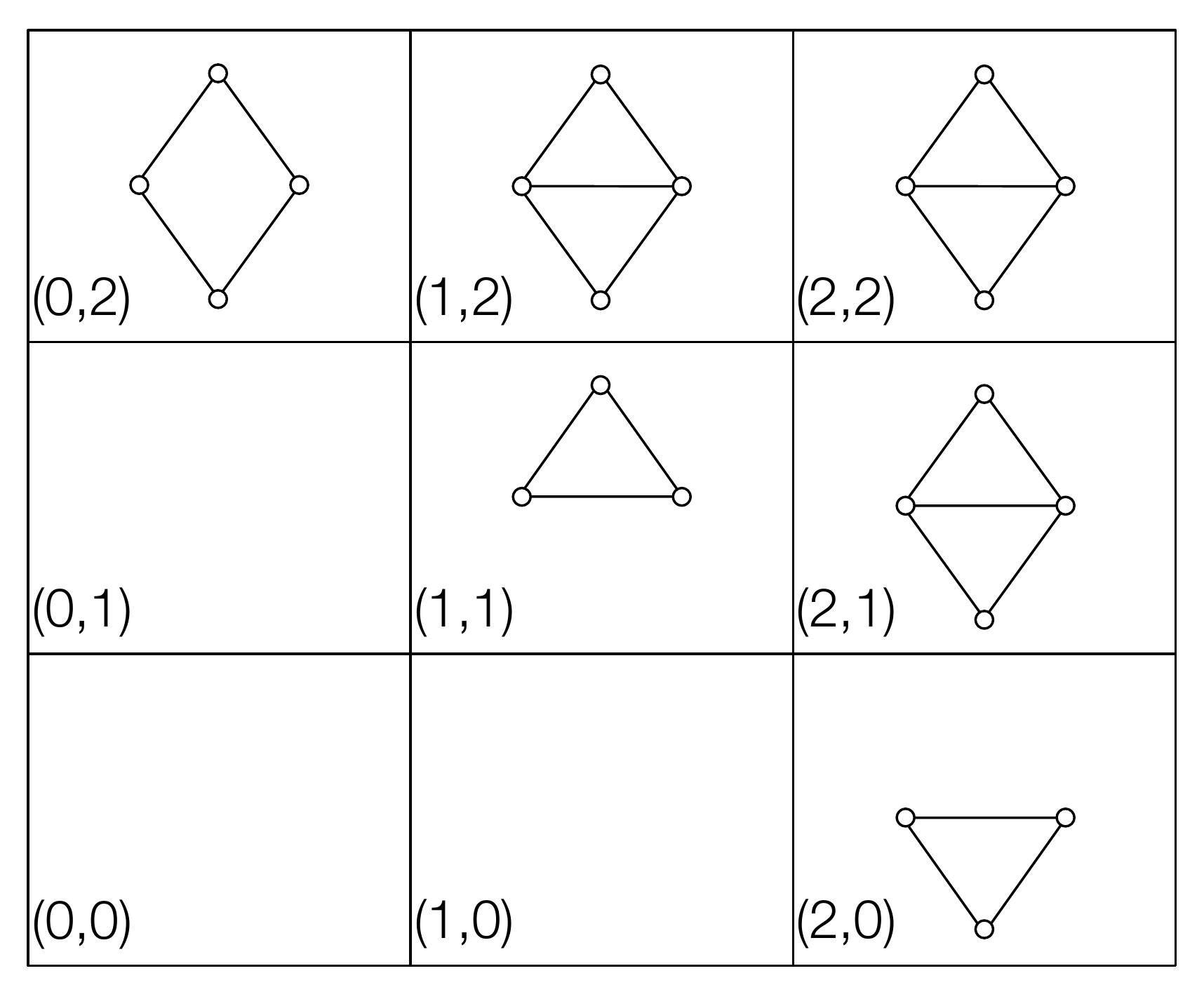}
                \caption*{$F:\N^r \rightarrow Spaces$}
                \label{ker}
        \end{subfigure}%
        \quad
        \quad
        \begin{subfigure}[b]{0.3\textwidth}
                \includegraphics[width=\textwidth]{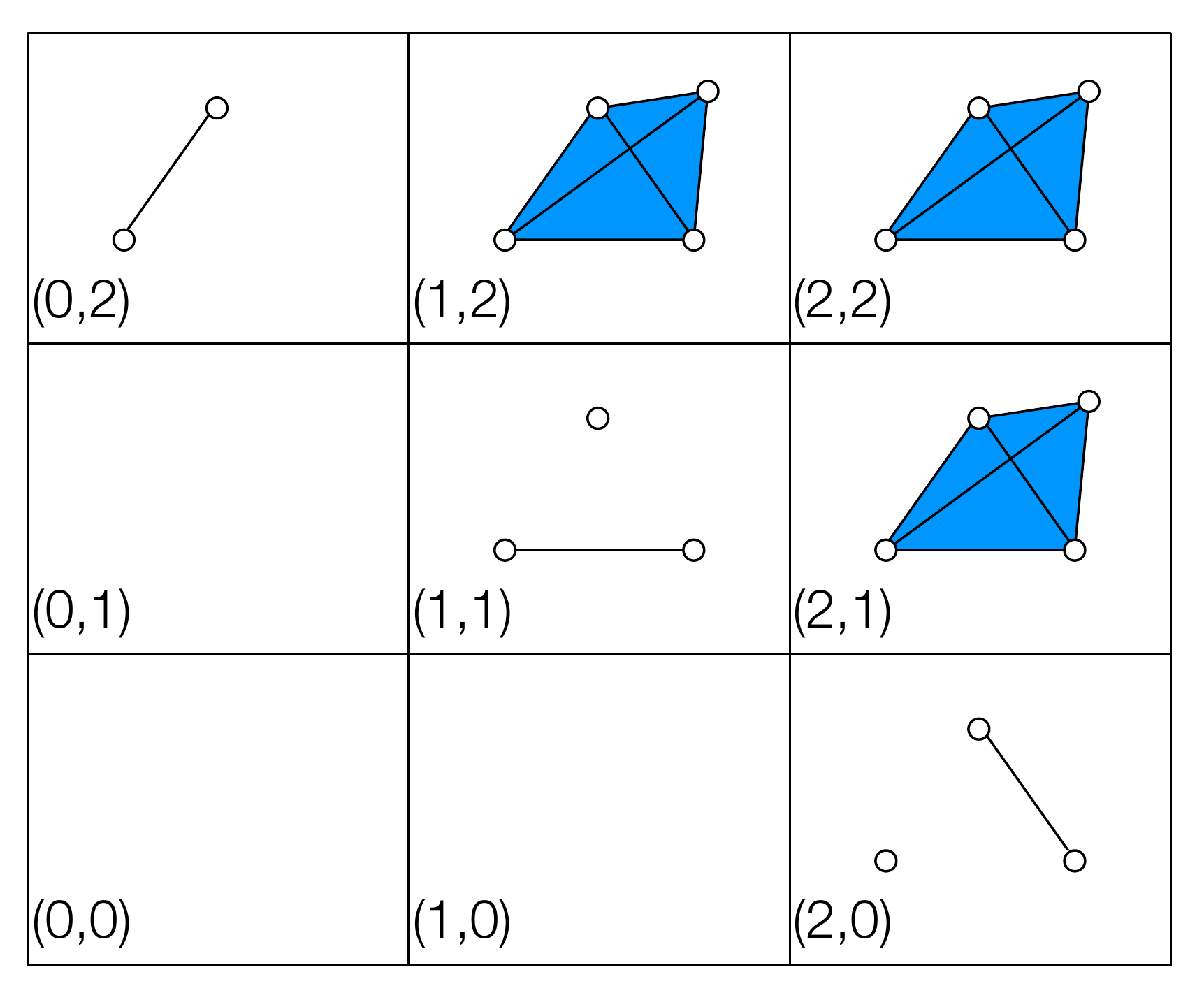}
                \caption*{$G:\N^r \rightarrow Spaces$}
                \label{coker}
        \end{subfigure}
        \caption{multifiltrations  with values in Spaces}\label{multifiltrations}
\end{figure}
In the multifiltration $F:\N^r \rightarrow \text{Spaces}$ there are no $2$-simplices and hence  $H_1(F,R) = \text{ker}(\partial_{1}\colon RF_{1}\to RF_{0})$.
On the square $\{ v\leq (2,2)\}\subset \N^2$, the functors $\text{ker}(\partial_{1}\colon RF_{1}\to RF_{0})$ and $\text{coker}(\partial_{2}\colon RG_{2}\to RG_{1})$ are given  respectively  by the diagrams:
 \[\xymatrix{
 R\rto^-{{\tiny
 \begin{pmatrix}
 1\\
 1
 \end{pmatrix}
 }}
  & R\oplus R \rto^-{\text{id}} & R\oplus R\\
 0\uto\rto & R\uto^-{{\tiny
 \begin{pmatrix}
 1\\
 0
 \end{pmatrix}
 }} \rto^-{{\tiny
 \begin{pmatrix}
 1\\
 0
 \end{pmatrix}
 }} & R\oplus R\uto_{\text{id}}\\
 0\rto\uto & 0\uto\rto & R\uto_-{{\tiny
 \begin{pmatrix}
 0\\
 1
 \end{pmatrix} \quad,\quad
 }}
 }
 \xymatrix{
 R\rto^-{{\tiny
 \begin{pmatrix}
 1\\
 0
 \end{pmatrix}
 }} & R\oplus R \rto^-{\text{id}} & R\oplus R\\
 0\uto\rto & R\uto^-{{\tiny
 \begin{pmatrix}
 0\\
 1
 \end{pmatrix}
 }} \rto^-{{\tiny
 \begin{pmatrix}
 0\\
 1
 \end{pmatrix}
 }} & R\oplus R\uto_{\text{id}}\\
 0\rto\uto & 0\uto\rto & R\uto_-{{\tiny
 \begin{pmatrix}
 -1\\
 -1
 \end{pmatrix}
 }}
 }
 \]
By Example \ref{notfromsets} both of the functors are not the $R$-span of any multifiltration of sets. 
\end{example}

We now assume that $F:\N^r \rightarrow \text{Spaces}$ is a compact multifiltration 
of spaces. It follows that $X=\text{colim}\, F$ is a finite complex. Since in general the functor $H_n(F,R)$ is not the $R$-span of a multifiltration of sets we cannot directly use the construction in Section~\ref{sec funmod} to compute a free presentation of the module $\boldsymbol{H_n(F,R)}$.  Instead our goal is to describe the module $\boldsymbol{H_n(F,R)}$ in such a way
that one can use very efficiently standard commutative algebra software or an   algorithm presented in~\cite{Multi2} which often is faster. As it was pointed out in~\cite{Multi2} this efficiency is a consequence of homogeneity and the fact that matrices involved are very simple. We proceed as follows:
\begin{enumerate}
\item  Consider  the decomposition $F_{n-1}=\coprod_{\sigma\in X_{n-1}}F_{n-1}[\sigma]$
(see~\ref{unique decomposition}). Define  $D_{n-1}:=\coprod_{\sigma\in X_{n-1}}\text{mor}_{\N^r}(0,-)$ and
  $\phi\colon F_{n-1}\to D_{n-1}$ to  be the coproduct of the 
unique inclusions $\coprod_{\sigma\in X_{n-1}}(F_{n-1}[\sigma]\hookrightarrow \text{mor}_{\N^r}(0,-))$. Note that $D_{n-1}$ is a free functor.
\item  Information about $F$ together with  the presentations and natural transformations given at the end of
Section~\ref{sec funsets} and in step ($1$) can be organized  into  the following commutative diagrams 
for any $0\leq i\leq n+1$ and $0\leq j\leq n$:
\[\xymatrix@C=35pt@R=30pt{
&\mathcal{G}F_{n+1}\rto^{p_{F_{n+1}}}\dto_{\overline{d_i}}& 
F_{n+1}\dto^{d_i}\\
\mathcal{K}F_n \ar@<0.7ex>[r]^-{\pi_0}\ar@<-.7ex>[r]_-{\pi_1}& \mathcal{G}F_{n}
\dto_{\alpha_j}
\rto^{p_{F_{n}}}
& 
F_{n}\dto^{d_j}\\
&D_{n-1} & F_{n-1}\ar@{_(->}[l]_-{\phi}
}\]
\item This leads to the following natural transformations:
\[\xymatrix@C=35pt@R=30pt{
 & \mathcal{G}F_{n+1}\ar@<1.7ex>[d]^-{\overline{d_{n+1}}}
\ar@{}@<0ex>[d]|-{\cdots}  \ar@<-1.7ex>[d]_-{\overline{d_{0}}} \\
\mathcal{K}F_n \ar@<0.7ex>[r]^-{\pi_0}\ar@<-.7ex>[r]_-{\pi_1}& \mathcal{G}F_{n}
\ar@<1.7ex>[d]^-{\alpha_{n}}
\ar@{}@<0ex>[d]|-{\cdots}  \ar@<-1.7ex>[d]_-{\alpha_{0}} \\
&   D_{n-1}
}\]
\item By applying the $R$-span functor and 
additivity we get two homomorphisms of $\N^r$-graded free $R[x_1,\ldots,x_r]$-modules:
\[\xymatrix@C=45pt{
\boldsymbol{R\mathcal{K}F_n}\oplus\boldsymbol{R\mathcal{G}F_{n+1}}\rrto^-{\left[ \boldsymbol{\pi_0}-\boldsymbol{\pi_1}\ 
\sum_{i=0}^{n+1}(-1)^i\boldsymbol{\overline{d_i}}\right]} & &\boldsymbol{R\mathcal{G}F_{n}}
\rrto^-{\sum_{j=0}^{n}(-1)^j \boldsymbol{\alpha_j}} && \boldsymbol{RD_{n-1}}
}\]
\end{enumerate}

\begin{prop}\label{prop homolgycomplexfree}
The composition of the above homomorphisms is trivial and the  homology of this complex is isomorphic to 
$\boldsymbol{H_n(F,R)}$.
\end{prop}
\begin{proof}
Consider the complex whose homology is $\boldsymbol{H_n(F,R)}$:
\[\xymatrix{\boldsymbol{RF_{n+1}}\rrto^{\sum_{i=0}^{n+1}(-1)^i\boldsymbol{d_i}}
& & \boldsymbol{RF_{n}}\rrto^{\sum_{i=0}^{n}(-1)^i\boldsymbol{d_i}} & & 
\boldsymbol{RF_{n-1}}
}\]
 Since $\phi\colon F_{n-1}\hookrightarrow D_{n-1}$ is an inclusion and
$p_{F_{n+1}}\colon \mathcal{G}F_{n+1}\to F_{n+1}$ is a surjection,
the bottom row of the following commutative diagram is also a complex whose 
 homology is  isomorphic to $\boldsymbol{H_n(F,R)}$:
\[\xymatrix@C=90pt{
\boldsymbol{R\mathcal{K}F_n}\oplus\boldsymbol{R\mathcal{G}F_{n+1}}
\dto|-{\text{projection}}
\rto^-{\left[ \boldsymbol{\pi_0}-\boldsymbol{\pi_1}\ 
\sum_{i=0}^{n+1}(-1)^i\boldsymbol{\overline{d_i}}\right]}  
&\boldsymbol{R\mathcal{G}F_{n}}\dto^{\boldsymbol{p_{F_n}}}
\rto^-{\sum_{j=0}^{n}(-1)^j \boldsymbol{\alpha_j}} & \boldsymbol{RD_{n-1}}\ar@{=}[d]\\
\boldsymbol{\mathcal{G}F_{n+1}}\rto^{\sum_{i=0}^{n+1}(-1)^i \boldsymbol{d_ip_{F_{n+1}}}}
&  \boldsymbol{RF_{n}}\rto^{\sum_{i=0}^{n}(-1)^i \boldsymbol{\phi d_i}}  & 
\boldsymbol{RD_{n-1}}
}\] 
Recall that   $\boldsymbol{RF_{n}}$ is the cokernel of the map 
$\boldsymbol{\pi_0}-\boldsymbol{\pi_1}\colon \boldsymbol{R\mathcal{K}F_n}\to \boldsymbol{R\mathcal{G}F_{n}}$. This implies the top row of the above diagram is also a complex whose homology is isomorphic to $\boldsymbol{H_n(F,R)}$ proving the  proposition.
\end{proof}

An important fact  is that the above  sequence  of free $\N^r$-graded $R[x_1,\ldots,x_r]$-modules
that computes $\boldsymbol{H_n(F,R)}$ can be easily and explicitly  described in terms of  the original multifiltration of spaces. Here are the involved modules:
\[\boldsymbol{R\mathcal{K}F_n}=\bigoplus_{\sigma\in X_n}\bigoplus_{v_0\not=v_1 \in \text{gen}(\sigma)}
<x^{\text{max}\{v_0,v_1\}}>\]
\[
\boldsymbol{R\mathcal{G}F_{n}}=\bigoplus_{\sigma\in X_n}\bigoplus_{v\in\text{gen}(\sigma)}<x^v>
\]
\[ \boldsymbol{RD_{n-1}}=\bigoplus_{\sigma\in X_{n-1}} R[x_1,\ldots, x_r]\]
and here is how to find the matrices
associated to the maps in this sequence (see~\ref{point polynomials} for our convention to describe homomorphisms between free 
$\N^r$-graded $R[x_1,\ldots,x_r]$-modules).
\begin{itemize}
\item Let  $\sigma $ be a simplex in $X_n$ or $X_{n+1}$. Consider the set 
$\{v\in \N^r\ |\  \sigma\in F(v)\}$. This is a saturated set and hence admits a finite 
minimal set of generators which we denote by $\text{gen}(\sigma)$. 
This set coincides with $\text{gen}(\text{supp}(F[\sigma]))$ and its elements are exactly  the minimal elements of 
the set $\{v\in \N^r\ |\  \sigma\in F(v)\}$.
\item The matrix $\left[ \boldsymbol{\pi_0}-\boldsymbol{\pi_1}\ 
\sum_{i=0}^{n+1}(-1)^i\boldsymbol{\overline{d_i}}\right]$ is  a concatenation of two matrices
one for  $\boldsymbol{\pi_0}-\boldsymbol{\pi_1}$ and one for $\sum_{i=0}^{n+1}(-1)^i\boldsymbol{\overline{d_i}}$. 
\item The matrix for $\boldsymbol{\pi_0}-\boldsymbol{\pi_1}$ is a block diagonal. The blocks are indexed by simplices in $X_{n}$ and the block corresponding to $\sigma$ in $X_{n}$
is of the size $ |\text{gen}(\sigma)|\times {|\text{gen}(\sigma)|\choose 2 }$. 
The entry in this block indexed by $v$ in  $\text{gen}(\sigma)$ and $v_0\not=v_1$ in $\text{gen}(\sigma)\choose 2$
has row grade $v$ and column grade $\text{max}\{v_0,v_1\}$. Its value is $1$ if $v=v_0$, $-1$ if $v=v_1$, and $0$ otherwise. 
\item The rows of the matrix   for $\sum_{i=0}^{n+1}(-1)^i\boldsymbol{\overline{d_i}}$
are indexed in the same way and have the same grades as the rows of the matrix   for $\boldsymbol{\pi_0}-\boldsymbol{\pi_1}$. The columns of the matrix   for $\sum_{i=0}^{n+1}(-1)^i\boldsymbol{\overline{d_i}}$ are divided into blocks indexed by simplices in  $X_{n+1}$. The columns in the block 
corresponding to $\sigma$ in $X_{n+1}$ are indexed by  $\text{gen}(\sigma)$. The corresponding element in  
$\text{gen}(\sigma)$ is the grade of the column.  
Each  column has exactly $n+2$ non-zero entries which are either $1$ or $-1$. For a column indexed by $v$ in  $\text{gen}(\sigma)$, the non-zero entries occur 
in the row blocks corresponding to the simplices $d_i(\sigma)$. In each such block there is only one non-zero entry and is equal to $(-1)^i$ and occurs in the row corresponding to the minimal element with respect to the lexicographical order in the set  $\{w\in \text{gen}(d_i(\sigma))\ | \ w\leq v\}$.

\item The matrix for $\sum_{j=0}^{n}(-1)^j \boldsymbol{\alpha_j}$ has rows indexed by simplices in $X_{n-1}$.
All the rows have grade  $0$. The columns are divided  into blocks indexed by simplices in $X_n$. 
The columns in the block 
corresponding to $\sigma$ in $X_{n}$ are indexed by  $\text{gen}(\sigma)$.
The corresponding element in  
$\text{gen}(\sigma)$ is the grade of the column.  
The entry in this matrix in the row indexed by $\tau$ in $X_{n-1}$ and the column indexed by 
 $v$ in $\text{gen}(\sigma)$ for $\sigma$Ê in $X_{n}$ has value $(-1)^i$ if $\tau=d_i(\sigma)$ and $0$ otherwise. 
Note that  in any row,  the entries  in the same  column block have the same value but different grades. 
\end{itemize}

We will now show our procedure to compute the module $\boldsymbol{H_1(F,R)}$ with an example.

\begin{example}
Consider the multifiltration  $F:\N^2\rightarrow \text{Spaces}$ which on the square $\{ v\leq (2,2)\}\subset \N^2$ is described in Figure~\ref{presentation example} and 
for $w$ in $\N^2\setminus \{ v\leq (2,2)\}$, the  maps induced by  $\text{min}(w,(2,2))\leq w$ are the identities.
The simplicial complex $X=\text{colim}\, F$ is given by the complex $F(2,2)$ and we choose an ordering of its vertices as indicated also in Figure~\ref{presentation example}.
\begin{figure}[ht]
\centering
\includegraphics[width=0.4\textwidth]{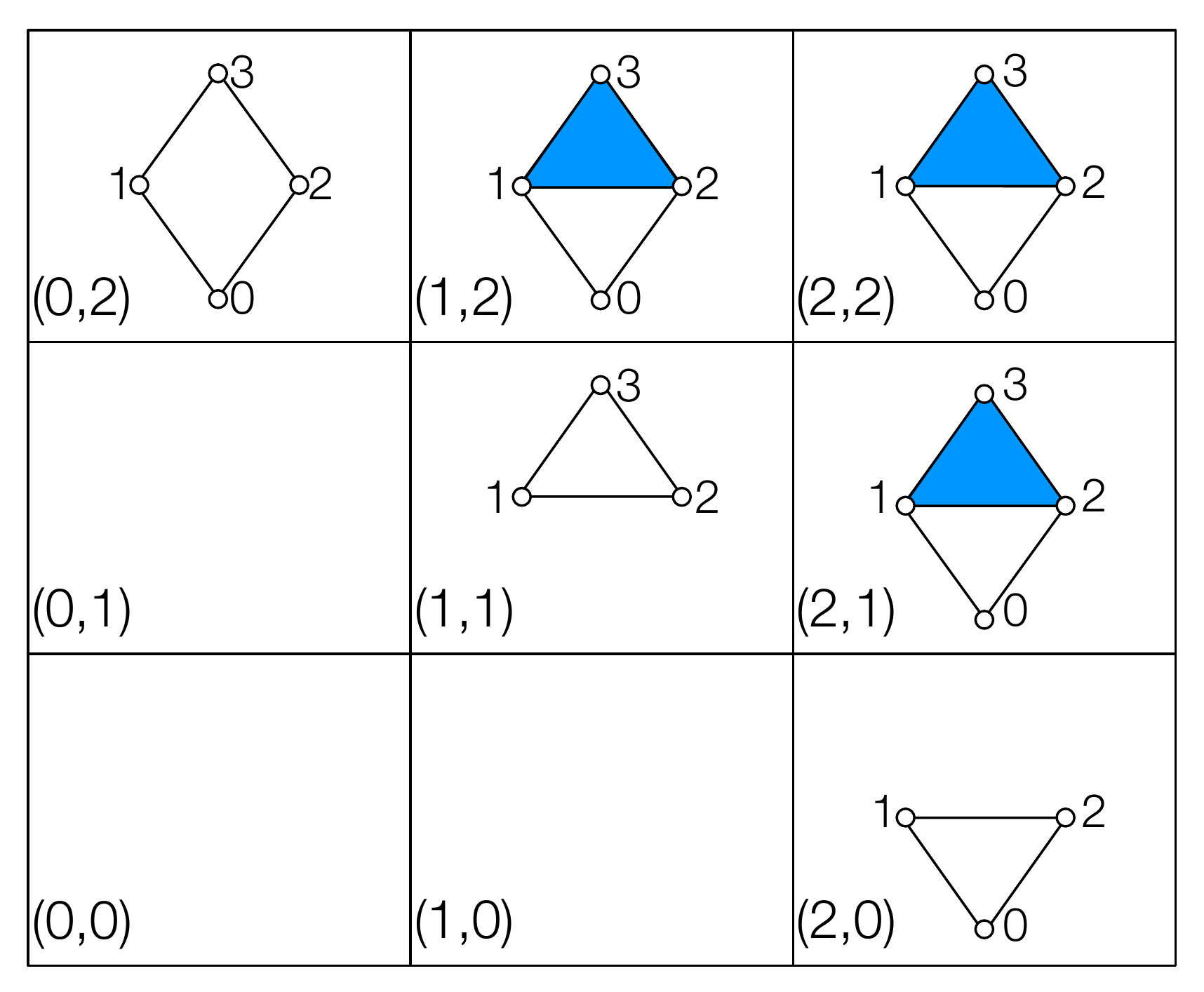}
\caption{}
\label{presentation example}
\end{figure}

The functors $H_1(F,R)\colon \N^2\to R\text{-Mod}$ on the square $\{ v\leq (2,2)\}\subset \N^2$ is given by the following commutative diagram:
\[
\xymatrix@C=15pt@R=15pt{
 R\rto^{1} & R\rto^{1} & R \\
 0\uto\rto & R\uto_{0} \rto^{0} & R\uto_{1} \\
 0\rto\uto & 0\uto\rto & R\uto_{1}
 }
 \]
We will now go through the steps presented above and construct the elements needed in Proposition \ref{presentation} to compute  $\boldsymbol{H_1(F,R)}$:

\begin{itemize}
\item $X_0=\{0,\ 1,\ 2,\ 3\}$, $X_1=\{0<1,\  0<2,\ 1<2,\  1<3,\  2<3\}$, \\$X_2=\{1<2<3\}$, and $X_n=\emptyset$ for $n\geq 3$.
\item For  an ordered simplex $\sigma$  in $X$, the minimal set of generators $\text{gen}(\sigma)$, ordered by the lexicographical order,   is given by the tables:
\medskip

\noindent \begin{tabular}[m]{|c|c|c|c|c|}
\hline
$\sigma$ & $0$ & $1$ & $2$ & $3$\\
\hline
$\text{gen}(\sigma)$ & (0,2) (2,0) &(0,2) (1,1) (2,0)& (0,2) (1,1) (2,0) & (0,2) (1,1) \\
\hline
\end{tabular}
\medskip

\noindent
\begin{tabular}[m]{|c|c|c|c|c|c|}
\hline
$\sigma$ & $0<1$ & $0<2$ & $1<2$ & $1<3$ & $2<3$ \\
\hline
$\text{gen}(\sigma)$ & (0,2) (2,0) &
(0,2) (2,0) & (1,1) (2,0) & (0,2) (1,1) &(0,2) (1,1)  \\
\hline
\end{tabular}
\medskip

\noindent
\begin{tabular}[m]{|c|c|}
\hline
$\sigma$ & $1<2<3$ \\
\hline
$\text{gen}(\sigma)$ &(1,2) (2,1)  \\
\hline
\end{tabular}
\medskip

\item We thus have:\\

\noindent
\begin{tabular}[m]{|c|c|}
\hline
$\boldsymbol{ R\mathcal{K}F_{1}}$ & $ 2\langle x^{(1,2)} \rangle \oplus  \langle  x^{(2,1)} \rangle\oplus 2\langle  x^{(2,2)}\rangle$ \\
\hline
$\boldsymbol{ R\mathcal{G}F_{1}}$ & $ 4\langle x^{(0,2)} \rangle \oplus 3\langle x^{(1,1)} \rangle\oplus 3 \langle  x^{(2,0)} \rangle$ \\
\hline
$\boldsymbol{ R\mathcal{G}F_{2}}$ & $\langle x^{(1,2)} \rangle \oplus \langle x^{(2,1)}\rangle$ \\
\hline
$\boldsymbol{ RD_{0}}$ & $4R[x_1,x_2]$\\
\hline
\end{tabular}
\medskip
\item
The matrix associated to $\boldsymbol{\pi_0-\pi_1}:\boldsymbol{R\mathcal{K}F_{1}}\to 
\boldsymbol{R\mathcal{G}F_{1}}$ with  the block decomposition and the column and row grades is given by: 
\[
\begin{blockarray}{ccc|c|c|c|c}
& & 0<1 &0<2 &1<2 &1<3 &2<3 \\
& &(2,2) &(2,2) & (2,1) & (1,2) &(1,2) \\
\begin{block}{cc(c|c|c|c|c)}
\multirow{2}{*}{$0<1$} &(0,2) & 1 & 0 & 0 & 0 & 0 \\
&(2,0) & -1 & 0 & 0 & 0 & 0\\
\cline{1-7}
\multirow{2}{*}{$0<2$} &(0,2) & 0 & 1 & 0 & 0 & 0 \\
&(2,0) & 0 & -1 & 0 & 0& 0 \\
\cline{1-7}
\multirow{2}{*}{$1<2$} &(1,1) & 0 & 0 & 1 & 0 & 0 \\
&(2,0) & 0 & 0 & -1 & 0 &0\\
\cline{1-7}
\multirow{2}{*}{$1<3$}&(0,2) & 0 & 0 & 0 & 1 & 0 \\
&(1,1) & 0 & 0 & 0 & -1 & 0 \\
\cline{1-7}
\multirow{2}{*}{$2<3$}&(0,2) & 0 & 0 & 0 & 0 & 1 \\
&(1,1) & 0 & 0 & 0 & 0 & -1\\
\end{block}
\end{blockarray}
 \]
\item
The matrix associated to $\sum_{j=0}^{2}(-1)^j \boldsymbol{\overline{d_i}}: \boldsymbol{R\mathcal{G}F_{2}}\to \boldsymbol{R\mathcal{G}F_{1}}$ with  the block decomposition and the column and row grades is given by: 
\[
\begin{blockarray}{cccc}
& &\BAmulticolumn{2}{c}{1<2<3 } \\
&&(1,2) &(2,1) \\
\begin{block}{cc(cc)}
\multirow{2}{*}{$0<1$} &(0,2) & 0 & 0  \\
&(2,0) & 0 & 0\\
\cline{1-4}
\multirow{2}{*}{$0<2$}&(0,2) & 0 & 0  \\
&(2,0) & 0 & 0\\
\cline{1-4}
\multirow{2}{*}{$1<2$}&(1,1) & 1 & 1 \\
&(2,0) & 0 & 0 \\
\cline{1-4}
\multirow{2}{*}{$1<3$}&(0,2) & -1 & 0  \\
&(1,1) & 0 & -1  \\
\cline{1-4}
\multirow{2}{*}{$2<3$}&(0,2) & 1& 0  \\
&(1,1) & 0 & 1\\
\end{block}
\end{blockarray}
 \]

\item

The matrix associated to $ \sum_{j=0}^{1}(-1)^j \boldsymbol{\alpha_j}:\boldsymbol{R\mathcal{G}F_{1}} \to \boldsymbol{RD_{0}}$ with the block decomposition and the column and row grades is given by:

\[
\begin{blockarray}{cccc|cc|cc|cc|cc}
& & \BAmulticolumn{2}{c|}{0<1} &  \BAmulticolumn{2}{c|}{0<2}  &  \BAmulticolumn{2}{c|}{1<2}  &
 \BAmulticolumn{2}{c|}{1<3}  &  \BAmulticolumn{2}{c}{2<3} \\
&&(0,2) &(2,0) & (0,2) & (2,0) &(1,1)&(2,0)&(0,2)&(1,1)&(0,2)&(1,1) \\
\begin{block}{cc(cc|cc|cc|cc|cc)}
0 &0 & -1& -1&-1& -1& 0 &0 &0&0&0&0 \\
\cline{1-12}
1 & 0 & 1 & 1 & 0& 0&-1& -1&-1&-1&0&0\\
\cline{1-12}
2 &0 & 0 & 0 & 1& 1& 1&  1&0&0& -1& -1\\
\cline{1-12}
3 & 0 & 0 & 0 & 0& 0& 0 & 0 &1&1&1&1\\
\end{block}
\end{blockarray}
 \]
\end{itemize}Ê

\end{example}

\section{Presentations of bifiltrations}\label{bifiltration}
Assume thatÊ $R$ Êis a field. For a general multifiltration of spaces, to get a presentation of its homology,
one can apply a standard algebra software to the exact sequence given in~\ref{prop homolgycomplexfree}.
In the case of a  bifiltration $F\colon\N^2\to\text{Spaces}$  one can try to be  more efficient.
Instead of  applying the software directly to the complex  given in~\ref{prop homolgycomplexfree}, one can first use 
 the fact that the polynomial ring $R[x_1,x_2]$ has the projective dimension $2$. This implies 
 that the kernel of any map between free modules is free. In particular  the kernel 
$\boldsymbol{Z}$ of the map 
$\sum_{j=0}^n (-1)^j\boldsymbol{\alpha_j}\colon \boldsymbol{R\mathcal{G}F_{n}}\to  \boldsymbol{RD_{n-1}}$
is free.  Let 
 $\phi\colon  \boldsymbol{R\mathcal{K}F_n}\oplus\boldsymbol{R\mathcal{G}F_{n+1}}\to  \boldsymbol{Z}$ 
be the map that fits into the following commutative diagram:
\[\xymatrix@C=50pt{
\boldsymbol{R\mathcal{K}F_n}\oplus\boldsymbol{R\mathcal{G}F_{n+1}}\rto^-{\phi}
\ar@/ _ 15pt/[rr]_-{\left[ \boldsymbol{\pi_0}-\boldsymbol{\pi_1}\ 
\sum_{i=0}^{n+1}(-1)^i\boldsymbol{\overline{d_i}}\right]}
 & \boldsymbol{Z}\ar@{^(->}[r] & 
\boldsymbol{R\mathcal{G}F_{n}}
}\]
The map  $\phi\colon  \boldsymbol{R\mathcal{K}F_n}\oplus\boldsymbol{R\mathcal{G}F_{n+1}}\to  \boldsymbol{Z}$   is  a free  presentation of  $\boldsymbol{H_n(F,R)}$.
 To take a full advantage of this  idea, one would need to be able to  describe in an efficient way a set of free generators
 of $\boldsymbol{Z}$. As of writing this paper, we have not found a method for doing it.

\section*{Acknowledgements} We would like to thank Bernd Sturmfels
for suggesting the problem of finding an efficient way to compute a presentation of multidimensional persistence modules.
 We would also like to thank Sandra Di Rocco for her advice and support and   
 Antonio Patriarca  for stimulating discussions.
 This paper would not be the
 same without valuable contributions of the referee.

\end{document}